\documentclass[usenames,dvipsnames,svgnames]{article}
\usepackage{amsmath,bm}
\numberwithin{equation}{subsection}
 
\usepackage{algorithm}
\usepackage{algpseudocode}
\usepackage{amssymb}
\usepackage{graphicx}
\usepackage{color}
\usepackage{cite}
\usepackage{amsthm}
\usepackage{url, verbatim}
\newtheorem{theorem}{Theorem}[section]
\newtheorem{lemma}[theorem]{Lemma}

\let\counterwithout\relax
\let\counterwithin\relax
\usepackage{chngcntr}
\counterwithout{equation}{subsection}
\counterwithin{equation}{section}

\theoremstyle{remark}

\usepackage[left=1.30in,right=1.30in,top=1.5in]{geometry}

\usepackage{bbm,mathabx}

\usepackage{array,multirow,booktabs} 

\usepackage{amsfonts}
\usepackage{algorithmicx}
\usepackage{algpseudocode}
\usepackage{epstopdf}
\usepackage{mathrsfs,mathtools}
\usepackage{enumitem}
\usepackage{chngcntr}

\usepackage[usenames,dvipsnames,svgnames]{xcolor}
\newcommand{\annote}[1]{{\leavevmode\color{Black}{#1}}}

\makeatletter
\def\BState{\State\hskip-\ALG@thistlm}
\makeatother

\newcommand{\bs}[1]{\boldsymbol{#1}}
\newcommand{\bmu}{\bs{\mu}}

\newcommand{\dx}[1]{\mathrm{d} #1}

\newcommand{\R}{\mathbbm{R}}

\newcommand{\calC}{\mathcal{C}}
\newcommand{\calR}{\mathcal{R}}
\newcommand{\calO}{\mathcal{O}}
\newcommand{\calQ}{\mathcal{Q}}

\newcommand{\calD}{\mathcal{D}}

\newcommand{\calL}{\mathcal{L}}
\newcommand{\calN}{\mathcal{N}}

\DeclareMathOperator*{\argmax}{argmax}

\newcommand{\naive}{na\"{\i}ve}

\begin{document}

\title{A robust {error estimator} and a residual-free {error indicator} for reduced basis methods}
\author{Yanlai Chen\thanks{
Department of Mathematics, University of Massachusetts Dartmouth, 285 Old Westport Road, North Dartmouth, MA 02747, USA. Emails: {\tt{\{jjiang, yanlai.chen\}@umassd.edu}}. J.~Jiang and Y.~Chen were partially supported by National Science Foundation grant DMS-1719698.
}
\and
Jiahua Jiang \footnotemark[1]
\and
Akil Narayan\thanks{
Scientific Computing and Imaging (SCI) Institute and Department of Mathematics, University of Utah, 72 S Campus Drive, Salt Lake City, UT 84112, USA. Email: {\tt{akil@sci.utah.edu}.} A.~Narayan is partially supported by NSF DMS-1720416 and AFOSR FA9550-15-1-0467.}
}

\date{}
\maketitle

\begin{abstract}
The Reduced Basis Method (RBM) is a rigorous model reduction approach for solving parametrized partial differential equations. It identifies a low-dimensional subspace for approximation of the parametric solution manifold that is embedded in high-dimensional space. A reduced order model is subsequently constructed in this subspace. RBM relies on residual-based error indicators or {\em a posteriori} error bounds to guide construction of the reduced solution subspace, to serve as a stopping criteria, and to certify the resulting surrogate solutions. Unfortunately, it is well-known that the standard algorithm for residual norm computation suffers from premature stagnation at the level of the square root of machine precision. 

In this paper, we develop two alternatives to the standard offline phase of reduced basis algorithms. First, we design a robust strategy for computation of residual error indicators that allows RBM algorithms to enrich the solution subspace with accuracy beyond root machine precision. Secondly, we propose a new error indicator based on the Lebesgue function in interpolation theory. This error indicator does not require computation of residual norms, and instead only requires the ability to compute the RBM solution. This residual-free indicator is rigorous in that it bounds the error committed by the RBM approximation, but up to an uncomputable multiplicative constant. Because of this, the residual-free indicator is effective in choosing snapshots during the offline RBM phase, but cannot currently be used to certify error that the approximation commits. However, it circumvents the need for \textit{a posteriori} analysis of numerical methods, and therefore can be effective on problems where such a rigorous estimate is hard to derive.
\end{abstract}

\section{Introduction}

The fundamental reason that many model reduction approaches for parametric partial differential equations (PDE) are successful is that, for many PDE of interest, the solution manifold induced by the parametric variation has small and rapidly-decaying Kolmogorov n-width \cite{MR774404}. Among the reduction strategies that utilize this fact is the {greedy approach} to the Reduced Basis Method (RBM). It identifies a small set of representative points in parameter space, and obtains solution to the PDE at these points. The construction of this point set proceeds via a greedy algorithm that relies on an \textit{a posteriori} error estimate for guidance. This solution ensemble on this small set is typically assembled from iterated queries to a potentially expensive existing solver.  

{The topic of this paper is the greedy approach for reduced basis methods.} In such cases, there are \textit{offline} and \textit{online} stages for RBM algorithms \cite{MR804937, MR616719, MR917452, MR2824231, MR3129759}, \annote{see also the recent monologue  {\cite{HesthavenRozzaStammBook, quarteroni2015reduced, benner2017model}}}. During the offline stage, significant computational effort is invested so that the online stage, when reduced order solutions for arbitrary parameter values are computed, can be efficient. During the offline stage, the parameter dependence is inspected and the greedy algorithm {which the RBM methods mainly rely on} judiciously selects a small number of parameter values on which the full-order, expensive PDE solver is employed to obtain so-called solution \textit{snapshots}. During the online stage, a surrogate solution is efficiently computed for any new parameter value as a linear combination of these stored snapshots. The coefficients of this linear combination are computed via a reduced-order formulation of the PDE. This reduced solve can usually be completed with orders of magnitude less effort than a full PDE solve, and thus RBM achieves significant speedup when both the offline phase is not too expensive and when multiple online queries are utilized.

A critical component in RBM algorithms is the \textit{a posteriori} error estimate, which dictates the adaptive sampling criterion in the offline greedy algorithm. This error estimate is the main concern of the current paper. The offline phase of the RBM algorithm finds a parameter value that maximizes the numerically-computed error estimate. Therefore, the accuracy (or the lack thereof) in this calculation dictates the accuracy of the RBM solution. 
Due to standard computational implementations of the RBM error estimate for elliptic PDE, the accuracy of this estimate stagnates around the level of the square root of the machine accuracy. Therefore, more accurate schemes for calculating the error estimate are necessary if one demands higher accuracy or when the query is close to the part of the parameter domain where resonances occur (and the stability constant approaches zero). 
To our knowledge, there \annote{are two} previous attempts to resolve this issue \cite{Casenave2012, Casenave2014_M2AN} \annote{and \cite{Ohlberger2014}}. 
\annote{The method in \cite{Casenave2012, Casenave2014_M2AN} employs} an additional sampling of the parameter domain, potentially randomly, to generate a linear system to solve online for the stable calculation of the \textit{a posteriori} error estimate. This is improved in \cite{Casenave2014_M2AN} by the empirical interpolation method. These approaches increase the Offline and Online cost, and may suffer from ill-conditioning depending on the additional sampling. 
\annote{On the other hand, {\cite{Ohlberger2014, canuto2009posteriori}} presents a strategy that amends a direct computation of the \textit{a posteriori} error estimate by rewriting it in a new form; the authors there show that their approach can circumvent stagnation errors due to floating-point arithmetic. {Thus, the loss of accuracy of half of the digits caused by taking the square root is not new. But with} 
this understanding, we devise a new approach for computing an \textit{a posteriori} error estimate, and demonstrate that it can be used to circumvent loss of significance from floating-point arithmetic. Our method performs similarly to \cite{Ohlberger2014}, but in cases where a matrix of residual vectors is rank-deficient, our approach is more efficient.} {We focus on the standard error criterion for selecting snapshots in this paper, that is, error in the solution, but note that goal-oriented strategies exist \cite{janon2016goal,hoang_efficient_2015,jiang_goal-oriented_2016}.}

The main contributions of this paper are twofold, in the theoretical and algorithmic design of robust residual-based, and residual-free error estimates for the offline RBM phase. Our first contribution is to the standard, residual-based method. We design and test a novel computational strategy for residual-based RBM error estimators that is capable of delaying error stagnation until much closer to machine precision. This new strategy computes the residual-based error norms in different ways compared to standard RBM algorithms, but are just as efficient as those algorithms.

Our second contribution to the offline RBM phase is more general and falls into a residual-free category. The efficiency of the error estimate calculation directly determines that of RBM. When the parameter is high-dimensional then the requisite size of the training set is very large, and it is computationally infeasible to repeatedly maximize a residual-based error estimate over the training set. The situation is exacerbated when a standard residual-based RBM error estimate cannot be computed, such as for hyperbolic problems, or in convection-dominated convection-diffusion equations. 
To the best of our knowledge, computational stratagems in the RBM framework to tackle this problem are underdeveloped. We propose and test an error indication strategy that forgoes the residual norm calculation entirely and requires only the RBM coefficients. This new procedure is rigorous and applicable to any parameterized problem without requiring any \textit{a posteriori} error analysis. However, the procedure cannot certify error due to the presence of a scaling constant that we have not been able to compute. Nevertheless, this new error estimate performs comparably to standard RBM algorithms for the examples that we have tried.

{All of the numerical examples in this manuscript investigate parametric PDE's that are relatively standard situations when RBM algorithms are known to perform well. We demonstrate for these cases that our strategies work well, and there is no methodological restriction that prevents our strategies from being applied in more general, difficult cases. However, we leave investigations of our approaches for computationally large-scale and more mathematically challenging parametric PDE's for future work.}

The remainder of this paper is organized as follows. In Section \ref{sec:background}, we provide a brief overview of the RBM framework and the typical setting it is successful for. Close attention is paid to the error estimate calculation, the focus of this paper. In Section \ref{sec:newapproaches}, we detail our two approaches to tackle the afore-mentioned two challenges. Section \ref{sec:results} is devoted to numerical results corroborating the efficiency and accuracy gain of the proposed approaches. Finally, some concluding remarks are given in Section \ref{sec:conclude}.

\section{Reduced basis method: a brief {overview}}
\label{sec:background}
In this section, we present a brief overreview of the Reduced Basis Method (RBM), in particular the error estimate and its implementation in the classical form. 
All of this is standard in the RBM literature. Therefore the reader familiar with RBM may skip this section, referring to Table \ref{tab:notation} for our notation.

\begin{table}
  \begin{center}
  \resizebox{\textwidth}{!}{
    \renewcommand{\tabcolsep}{0.4cm}
    \renewcommand{\arraystretch}{1.3}
    {\scriptsize
    \begin{tabular}{@{}lp{0.8\textwidth}@{}}
      \toprule
      $\bmu$ & Parameter in $\calD \subseteq \R^p$ \\
      $u(\bmu)$ & Function-valued solution of a parameterized PDE \\
      $\mathcal{N}$ & Degrees of freedom (DoF) in PDE ``truth" solver \\
      $X^\calN$ & Truth solver solution space, having dimension $\calN$\\
      $u^{\mathcal{N}}(\bmu)$ & Truth solution (finite-dimensional)\\
      $N$ & Number of reduced basis snapshots, $N \ll \mathcal{N}$\\
      $\bmu^j$ & ``Snapshot" parameter values, $j=1, \ldots, N$\\
      $S^N$ & Sample set $S^N = \{\bmu^1, \dots, \bmu^N\}$\\
      $X^{\mathcal{N}}_{{N}}$ & Span of $u^{\mathcal{N}}\left(\bmu^k\right)$ for $k=1, \ldots, N$\\
      $u_N^{\mathcal{N}}(\bmu)$ & Reduced basis solution, $u_N^{\mathcal{N}} \in X^{\mathcal{N}}_{{N}}$\\
      $e_N(\bmu)$ & Reduced basis solution error, equals $u^{\mathcal{N}}(\bmu) - u_N^{\mathcal{N}}(\bmu)$ \\
      $\Xi_{\rm{train}}$ & Parameter training set, a finite subset of $\mathcal{D}$ \\
      $\Delta_{{N}} \left(\bmu\right)$ & Error estimate (upper bound) for $\left\|e_N\left(\bmu\right)\right\|$ \\
      $\epsilon_{\mathrm{tol}}$ & Error estimate stopping tolerance in greedy sweep \\
    \bottomrule
    \end{tabular}
  }
    \renewcommand{\arraystretch}{1}
    \renewcommand{\tabcolsep}{12pt}
  }
  \end{center}
  \caption{Notation used throughout this article.}\label{tab:notation}
\end{table}

\subsection{Problem setting}
Let $\mathcal{D} \subset \mathbb{R}^{p}$ be the domain for a $p$-dimensional parameter, and let $\Omega \subset \mathbb{R}^{d}$ (for $d = 2 ~ \text{or} ~ 3$) be a bounded physical domain. Let $X$ be a Hilbert space of functions on $\Omega$. Given $\bm{\mu} \in \mathcal{D}$, the goal is to compute $u(\bm{\mu}) \in X$ satisfying
\begin{equation}
\label{eq:sat}
a(u(\bm{\mu}),v; \bm{\mu}) = f(v; \bm{\mu}), \quad v \in X,
\end{equation}
which corresponds to a parametric partial differential equation (pPDE) written in weak form; {$a(\cdot, \cdot; \bmu)$} is a bilinear form and $f$ may encode forcing terms and/or boundary conditions. We assume $u \in X$ and that $H^{1}_{0}(\Omega) \subset X(\Omega) \subset H^{1}(\Omega)$, where
\begin{align*}
  H^1(\Omega) &= \left\{ w \in L^2(\Omega) \; \big| \; \int_\Omega \left\| \nabla w \right\|^2 \dx{x} < \infty \right\}{,} \\
  H^1_0(\Omega) &= \left\{ w \in L^2(\Omega) \; \big| \; \int_\Omega \left\| \nabla w \right\|^2 \dx{x} < \infty \textrm{ and } w|_{\partial \Omega} = 0 \right\}.
\end{align*}
We denote by $(\cdot , \cdot)_{X}$ the inner product associated with the space $X$, whose induced norm $|| \cdot ||_{X} = \sqrt{(\cdot , \cdot)_{X}}$ is equivalent to the usual $H^{1}(\Omega)$ norm. For $\bmu$-uniform well-posedness of \eqref{eq:sat}, we assume that $a(\cdot ,\cdot; \bm{\mu}): X \times X \rightarrow \mathbb{R}$ is continuous and uniformly coercive over $X$ for all $\bm{\mu}$ in $\mathcal{D}$, and that $f(\cdot)$ is a  linear continuous functional over $X$ for all $\bmu$. That is,
\begin{subequations}
\begin{align}\label{eq:con}
  \gamma(\bm{\mu}) &\coloneqq \underset{w \in X}{\sup} ~ \underset{v \in X}{\sup}~ \frac{a(w, v; \bm{\mu})}{||w||_{X}||v||_{X}} < \infty, & \forall \bm{\mu} &\in \mathcal{D},  
\\\label{eq:coe}
\alpha(\bm{\mu}) &\coloneqq \underset{w \in X} \inf \frac{a(w, w; \bm{\mu})}{||w||^2_{X}} \geqslant \alpha_{0} > 0, & \forall \bm{\mu} &\in \mathcal{D}, \\\label{eq:coe_rev}
\sup_{u \in X} \left| f(u;\bmu) \right| &< \infty, & \forall \bmu &\in \mathcal{D}.
\end{align}
\end{subequations}
As is common in the RBM literature \cite{Rozza_Huynh_Patera}, we assume that $a(\cdot, \cdot; \bm{\mu})$ is ``affine'' with respect to the parameter $\bm{\mu}$: I.e., there exist $\bm{\mu}$-dependent coefficient functions $\Theta_{a}^{q}: \calD \rightarrow \mathbb{R}$ for $q = 1, \ldots Q_a$, and $\Theta_f^q: \calD \rightarrow \R$ for $q = 1, \ldots, Q_f$, and corresponding continuous $\bmu$-independent bilinear forms $a^{q}(\cdot, \cdot): X \times X \rightarrow \R$ and linear forms {$f^{q}(\cdot): X \rightarrow \R$}, respectively, such that 
\begin{align}
\label{eq:assum_a}
a(w, v; \bm{\mu}) &= \sum_{q = 1}^{Q_{a}} \Theta_{a}^{q}(\bm{\mu})a^{q}(w,v), &
f(w; \bm{\mu}) &= \sum_{q = 1}^{Q_{f}} \Theta_{f}^{q}(\bm{\mu})f^{q}(w).
\end{align}
There are strategies for situations when the affine assumption is not satisfied, e.g., empirical interpolation \cite{BarraultMaday2004}. {These strategies generally replace a non-affine operator by an affine operator with sufficiently large $Q_a$ and $Q_f$ so that the solution of the affine problem well approximates the solution to the non-affine problem. Note, however, that such an approach can require large $Q_a$ and/or $Q_f$ and that RBM strategies can suffer substantially in computational efficiency in this case.}

In order to compute solutions to \eqref{eq:sat} suppose that for $\bmu$ fixed, a spatial discretization (e.g., a finite element solver) can be employed with $\mathcal{N} \gg 1$ degrees of freedom that computes an approximate solution $u^\calN(\bmu)$ to within an acceptable accuracy for every $\bmu \in \mathcal{D}$. This approximate solution $u^\calN$ is called the ``truth" solution, and the associated spatial discretization and solver is called the ``truth" solver. 

The truth solution $u^{\mathcal{N}}$ is {sought in} an $\calN$-dimensional subspace $X^{\mathcal{N}}$ (i.e., $\dim(X^{\mathcal{N}}) = \mathcal{N}$) that is a discretization of $X$, and \eqref{eq:sat} is discretized as
\begin{equation}\label{eq:update_problem}
\begin{cases}
\text{For } \bm{\mu} \in \mathcal{D},~ \text{find the ``truth'' approximation } u^{\mathcal{N}}(\bm{\mu}) \in X^{\mathcal{N}} ~ \text{such that} \\
a(u^{\mathcal{N}}, v; \bm{\mu}) = f(v; \bm{\mu}) \quad \forall v \in X^{\mathcal{N}}.
\end{cases}
\end{equation}
The other relevant quantities are defined according to the discretization. For example, the coercivity constant \eqref{eq:coe} is approximated by its numerical couterpart $\alpha^{\mathcal{N}}(\bm{\mu})  = {\displaystyle \inf_{w \in X^{\mathcal{N}}}} \frac{a(w, w; \bm{\mu})}{||w||^{2}_{X}}, ~ \forall \bm{\mu} \in \mathcal{D}.$

A \naive{} approach to computing solutions to \eqref{eq:sat} for many $\bmu$ would be to query the truth solver many times, which is expensive under the $\calN \gg 1$ assumption. RBM methods attempt to provide numerical solutions of \eqref{eq:sat} with accuracy comparable to $u^\calN$, but with orders-of-magnitude less computational cost than the truth solver. The essential idea is to project the collection of discrete solutions $u^\calN(\bmu)$ for $\mu \in \mathcal{D}$ onto a low-dimensional representation, and then to efficiently compute this projected representation.

\subsection{RBM framework}

The reduced basis method is a reliable model reduction strategy that replaces  a relatively expensive truth solver \eqref{eq:update_problem} with a less expensive surrogate. 
The best possible accuracy of the surrogate is governed by a theoretical quantity, the Kolmogorov $N$-width of the solution set $U^\calN$, defined as
\begin{align}\label{eq:solution-set}
  U^\calN \coloneqq \left\{ u(\bmu) \;\; |\;\; \bmu \in \mathcal{D} \right\} \subset X^{\mathcal{N}}.
\end{align}
When the $N$-width of $U^\calN$ decays quickly with respect to $N$, an RBM strategy is effective. Practitioners in advance identify a large but finite training set discretizing the parameter domain $\Xi_{\rm train} \subset \mathcal{D}$, and a maximum dimension $N_{\rm max}$ (usually $\ll \mathcal{N}$). An RBM algorithm then approximates the solution space by an $N$-dimensional subspace of $X^{\mathcal{N}}$, denoted by $X^{\mathcal{N}}_{N}$, with $N \leq N_{\rm max}$. 

The surrogate for the truth discretization is denoted $u_N^\calN (\bmu)$ and is computed as a member of  the reduced space $X^{\mathcal{N}}_{N}$.  The space $X^{\mathcal{N}}_{N}$ is constructed in a hierarchical manner as the span of so-called ``snapshots'', by hierarchically constructing a sample set {$S^N = \{\bmu^1, \dots, \bmu^N\}$} from the training set $\Xi_{\rm train}$ and solving \eqref{eq:update_problem} with $\bm{\mu} = \bm{\mu}^{n}$. Explicitly:
\begin{equation}\label{eq:space_define}
X^{\mathcal{N}}_{N} \coloneqq \text{span} \{ u^{\mathcal{N}}(\bm{\mu}^{n}), 1 \leq n \leq N \}, \quad N = 1, \dots, N_{\rm max}.
\end{equation}

Given $\bm{\mu} \in \mathcal{D}$, we define the RBM solution $u_{N}^{\mathcal{N}}(\bmu) \in X^{\mathcal{N}}_{N}$ as the solution to the following {\em reduced} problem
\begin{equation}
\label{eq:reduced_system}
\begin{cases}
\text{For } \bm{\mu} \in \mathcal{D},~ \text{find the RB solution } u_{N}^{\mathcal{N}}(\bm{\mu}) \in X_{N}^{\mathcal{N}} \subset X^{\mathcal{N}}~ \text{such that } \\
a(u_{N}^{\mathcal{N}}, v; \bm{\mu}) = f(v; \bmu) \quad \forall v \in X_{N}^{\mathcal{N}}.
\end{cases}
\end{equation} 
The truth system \eqref{eq:update_problem} is $\mathcal{N}$-dimensional, but the reduced system \eqref{eq:reduced_system} is $N$-dimensional. When $N \ll \mathcal{N}$, this results in a significant computational savings. 
This saving is realized by precomputing and storing the parameter-independent components of the RB ``stiffness'' matrix which is decomposed via the affine assumption \eqref{eq:assum_a}
\begin{equation}\label{eq:rbstiff}
\left(a(v_m , v_n ; \bm{\mu})\right)_{N \times N} = \sum_{q = 1}^{Q_a} \Theta_a^q(\bmu) \left( a^q(v_m , v_n) \right)_{N \times N}. 
\end{equation}
The nested structure of $X^{\mathcal{N}}_{N}$ ($X^{\mathcal{N}}_{1} \subset X^{\mathcal{N}}_{2} \subset \dots \subset X^{\mathcal{N}}_{N_{\rm max}} \subset X^{\mathcal{N}}$) allows us to expand these parameter-independent components  
$\left( a^q(v_m , v_n) \right)_{N \times N}$ by adding one more row and one more column each time a new sample location $\bmu^i$ is selected and the new snapshot resolved.

{
  \subsection{Computational details}
The Galerkin procedure in \eqref{eq:reduced_system} can be transformed into an algorithm by using the following ansatz for the RBM solution:
\begin{equation}\label{eq:rbsolution}
  u_{N}^{\mathcal{N}}(\bm{\mu}) = \sum_{m = 1}^{N} u_{Nm}^{\mathcal{N}}(\bm{\mu}) u^{\mathcal{N}}\left(\bmu^{{m}}\right).
\end{equation}
Using this in \eqref{eq:reduced_system} and choosing the test functions as $v = u^{\mathcal{N}}\left(\bm{\mu}^{(n)}\right)$ for $n = 1, \ldots, N$, results in a solvable linear system for the unknown RB coefficients $\{u_{Nm}^{\mathcal{N}}(\bm{\mu})\}_{m = 1}^{N}$, which defines $u_N^{\mathcal{N}}$. In practice, this strategy tends to produce ill-conditioned systems. To ameliorate this effect, practitioners usually choose an orthonormal basis for $X_N^{\mathcal{N}}$ for both trial and test functions:
\begin{align}\label{eq:rbsolution-xi}
  u_{N}^{\mathcal{N}}(\bm{\mu}) &= \sum_{m = 1}^{N} \widehat{u}_{Nm}^{\mathcal{N}}(\bm{\mu}) \xi_m, & \left\langle \xi_m, \xi_n \right\rangle_X &= \delta_{n,m}, & \left\{ \xi_m \right\}_{m=1}^N \subset X_N^{\mathcal{N}}.
\end{align}
The $\xi_n$ are hierarchically computed by orthogonalizing the snapshots $u^{\mathcal{N}}\left(\bm{\mu}^{(m)}\right)$ each time $S^N$ is updated. We note that computing the $\widehat{u}_{N m}^\calN$ coefficients is mathematically equivalent to computing the $u_{N m}^\calN$ coefficients, as the two are related by a change of basis transformation: The computational representation \eqref{eq:rbsolution-xi} can be transformed into the representation \eqref{eq:rbsolution} (and vice versa) through a linear transformation that we omit for brevity. The representation \eqref{eq:rbsolution-xi} is used in practical implementations of RBM algorithms (both in general and in the examples of this paper), but our discussion will use the formulation \eqref{eq:rbsolution} because this formulation is more convenient to describe our residual-free error indicator in Section \ref{sec:ee3}. 
}

\subsection{Selecting snapshots through the {\em a posteriori} error estimate}

Here we describe the procedure for  selecting the representative parameters $\bmu^1, \ldots, \bmu^N$ for the sample set $S^N$. This is an important task since these parameter choices define the reduced space \eqref{eq:space_define}. 
RBM adopts a greedy scheme to iteratively construct $S^N$, and leans on mathematical theory for the truth discretization \eqref{eq:update_problem}. In particular, there exist efficiently-computable error estimates that quantify the discrepancy between the dimension-$n$ RBM surrogate solution $u^\calN_n(\bmu)$ and the truth solution $u^\calN(\bmu)$.  This error estimate is denoted $\Delta_n$, and satisfies $\Delta_n(\bmu) \geq \left\| u^\calN_n(\bmu) - u^\calN(\bmu)\right\|_{X^{\calN}}$. Assuming existence of this error estimate, the greedy procedure for constructing $S^N$ is summarized in Algorithm \ref{alg:rbmgreedy}.

\begin{algorithm}[H] 
\begin{algorithmic}[1]
  \State Input: training set $\Xi_{\rm train}$, an accuracy tolerance $\varepsilon_{\mathrm{tol}}$, maximum RB dimension $N_{\mathrm{max}}$.
\State Randomly select the first sample $\bmu^1 \in \Xi_{\rm train}$, and set $n = 1$ and $\varepsilon = 2 \varepsilon_{\mathrm{tol}}$.
\State Obtain truth solution $u^\mathcal{N}(\bmu^1)$, and set $X^\calN_1 = \mbox{span}\left\{u^{\mathcal N}(\bmu^1)\right\}$.
\While {$(\varepsilon > \varepsilon_{\mathrm{tol}} \textrm{ and } n < N_{\mathrm{max}})$}
\vspace{0.05in}
\For{each $\boldsymbol{\mu} \in \Xi_{\rm train}$}
\State Obtain RBM solution $u^{\mathcal{N}}_{n}(\boldsymbol{\mu}) \in X^\calN_n$ and error estimate ${\Delta_{n}}(\boldsymbol{\mu})$
\EndFor
\vspace{0.05in}
\State $\bmu^{n+1} = \underset{\bm{\mu} \in \Xi_{\rm train}}{\argmax} \Delta_{n}(\bm{\mu})$, $\varepsilon = \Delta_{n}(\bmu^{n+1})$.
\State Augment the sample set $S^{n+1} = S^n \bigcup \{\bmu^{n+1}\}$ and the RB space $X^\calN_{n+1} = X^\calN_n \oplus \{u^\calN(\bmu^{n+1})\}$.
\State Set $n \leftarrow n+1$.
\vspace{0.05in}
\EndWhile
\end{algorithmic}
\caption{Greedy algorithm for constructing $S^N$ and $X_N^\calN$.}\label{alg:rbmgreedy}
\end{algorithm}

The design and efficient implementation of the error bound $\Delta_n$ is usually accomplished with {\em a posteriori} error estimates from the truth discretization. Mathematical rigor of this estimate is crucial for the accuracy of the reduced basis solution. The main novelties of this paper are replacements of the ``classical" \textit{a posteriori} error estimate with alternatives that are either more efficient or have enhanced accuracy properties. We finish this section by describing how $\Delta_{n}$ is evaluated in the classical fashion. 

The error between the reduced basis surrogate solution and the truth discretization is $e_N(\bm{\mu}) := u^{\mathcal{N}}(\bm{\mu}) - u^{\mathcal{N}}_{N}(\bm{\mu}) \in X^{\mathcal{N}}$. Unfortunately, this is not computable directly without knowledge of $u^\calN$, which we want to avoid computing. However, the linearity of $a(\cdot, \cdot; \bmu)$ results in
\begin{equation}
\label{eq:residual_eq}
a(e_N(\bm{\mu}), v; \bm{\mu}) = r_N(v; \bm{\mu}) \quad \forall v \in X^\calN,
\end{equation}
with the residual $r_N(v; \bm{\mu}) \in (X^{\mathcal{N}})'$ (the dual of $X^{\calN}$) defined as $f(v; \bm{\mu}) - a(u_{N}^{\mathcal{N}}(\bm{\mu}), v; \bm{\mu})$. The Riesz representation theorem and a variational inequality imply that $\lVert e_N(\bm{\mu})\rVert_{{X^{\mathcal{N}}}} \le \frac{\lVert r_N(\cdot; \bm{\mu})\rVert_{(X^{\mathcal{N}})'}}{\alpha^{\mathcal{N}}(\bm{\mu})}$, 
where $\alpha^{\mathcal{N}}(\bm{\mu}) = \underset{w \in X^{\mathcal{N}}}{\inf} \frac{a(w, w, \bm{\mu})}{||w||^2_{X}}$ is the stability (coercivity) constant for the elliptic bilinear form $a$. Therefore, we can define an {\em a posteriori} RBM error estimator as
\begin{equation}\label{eq:error_estimator_en}
  \Delta_{N}(\bm{\mu}) = \frac{\lVert r_N(\cdot; \bm{\mu})\rVert_{(X^{\mathcal{N}})'}}{\alpha^{\mathcal{N}}_{LB}(\bm{\mu})} \geq \left\| e_N(\bmu)\right\|_{X^\calN}{.}
\end{equation}
Here $\alpha^{\mathcal{N}}_{LB}(\bm{\mu})$ is a lower bound for $\alpha^{\mathcal{N}}(\bm{\mu})$ which is expensive to compute directly for all $\bmu$. However, {approaches for computationally efficient estimation of the stability factor $\alpha^{\mathcal{N}}_{LB}(\bm{\mu})$ has been undergoing vast development in \cite{HuynhSCM, HKCHP} and recently in \cite{Chen2015_NNSCM, yano2014space, manzoni2015heuristic}. Furthermore, an offline-online decomposition is also exploited to speed up the evaluation of $\alpha^{\mathcal{N}}_{LB}(\bm{\mu})$.}

The remaining ingredient for efficient computation of the {\em a posteriori} error estimation is the evaluation of $\lVert r_N(\cdot; \bm{\mu})\rVert_{(X^{\mathcal{N}})'}$, which is the chief concern of this manuscript. The following discussion details how this is achieved in a standard RB implementation. The Riesz representation theorem states that we can calculate functions $C^{{\tilde{q}}} \in X^\calN$ and $\mathcal{L}_{m}^{q} \in X^\calN$, for ${1 \le \tilde{q} \le Q_{f},}  1\le m \le N, 1 \le q \le Q_{a}$, such that 
\begin{equation}\label{error_es_problem}
\begin{cases}
  (\mathcal{C}^{{\tilde{q}}}, v)_{X^\calN} = f^{{\tilde{q}}}(v)_{X^{\calN}} \quad \forall v \in X^{\mathcal{N}} \\
  (\mathcal{L}_{m}^{q}, v)_{X^{\calN}} = a^{q}(u^{\calN}\left(\bmu^m\right), v) \quad \forall v \in X^{\mathcal{N}}.
\end{cases}
\end{equation}
With the availability of $\calC^{{\tilde{q}}}$ and $\calL_m^q$, the classical implementation of RBM then adopts an {offline-online} decomposition for $\lVert r_N(\cdot; \bm{\mu})\rVert_{(X^{\mathcal{N}})'}^2$. Combining \eqref{eq:residual_eq}, \eqref{eq:rbsolution}, and \eqref{eq:assum_a}, we have 
\begin{align}\nonumber
  \lVert r_N(\cdot; \bm{\mu})\rVert_{(X^{\mathcal{N}})'}^2 &= 
\sum_{{\tilde{q}_1} = 1}^{Q_f}  \sum_{{\tilde{q}_2} = 1}^{Q_f} \Theta_f^{{\tilde{q}_1}}(\bmu) \Theta_f^{{\tilde{q}_2}}(\bmu) (\mathcal{C}^{{\tilde{q}_1}}, \mathcal{C}^{{\tilde{q}_2}})_{X^{\calN}} + \\
&\sum_{q = 1}^{Q_{a}}\sum_{m = 1}^{N}\Theta_a^{q}(\bm{\mu})u^{\mathcal{N}}_{Nm} \left\{ \sum_{q' = 1}^{Q_{a}}\sum_{m' = 1}^{N}\Theta_a^{q'}(\bm{\mu})u^{\mathcal{N}}_{Nm'}(\mathcal{L}_{m}^{q}, \mathcal{L}_{m'}^{q'})_{X^{\calN}}\right\} \nonumber\\
\label{eq:error_es_quan}
&- 2\sum_{q = 1}^{Q_{a}}\sum_{m = 1}^{N} {\sum_{\tilde{q} = 1 }^{Q_f}} \Theta_a^{q}(\bm{\mu})u^{\mathcal{N}}_{Nm}(\bm{\mu})({\mathcal{C}^{\tilde{q}}}, \mathcal{L}_{m}^{q})_{X} .
\end{align}
Most of the quantities above can be precomputed and stored explicitly from the form of the pPDE.  Therefore, the offline stage is devoted to calculating and storing 
\begin{equation*}
(\mathcal{C}^{{\tilde{q}_1}}, \mathcal{C}^{{\tilde{q}_2}})_{X^{\calN}}, (\mathcal{C}{^{\tilde{q}_1}},\mathcal{L}_{m}^{q})_{X^{\calN}}, (\mathcal{L}_{m}^{q}, \mathcal{L}_{m'}^{q'})_{X^{\calN}}, \,\, 1 \le m, m' \le N_{\rm RB}, 1 \le {{\tilde{q}_1}}, {{\tilde{q}_2}} \le Q_{f}, 1 \le q, q' \le Q_{a}. 
\end{equation*}
During the online stage, given any parameter $\bm{\mu}$, we only need to evaluate $\Theta_a^{q}(\bm{\mu}), \Theta_f^{{\tilde{q}}}(\bmu), u^{\mathcal{N}}_{Nm}(\bm{\mu}), 1 \le m \le N$, and perform the addition and subtraction according to \eqref{eq:error_es_quan}. However, the coefficients $u^\calN_{N m}$ are the RBM expansion coefficients from \eqref{eq:rbsolution}, and therefore evaluation of \eqref{eq:error_es_quan} at each $\bmu$ also requires computation of the RB solution $u^\calN_N$. 
We denote this numerical approximation of $\Delta_n(\bmu)$ in \eqref{eq:error_estimator_en} by ${\mathcal E}_1(\bmu; n)$.

\section{Novel approaches for error quantification}
\label{sec:newapproaches}

As shown by Algorithm \ref{alg:rbmgreedy}, the classical RBM computes the maximum of the error estimator $\Delta_n(\cdot)$ over $\Xi_{\rm train}$ in a brute-force manner. $\Xi_{\rm train}$ is large especially for a high-dimensional parameter domain. Moreover, this maximization procedure must be done for every $n =1, \ldots, N$.
For these two reasons, the process of selecting snapshots is usually the computational bottleneck of RBM algorithms. Another observation is that the dual norm of the residual is computed as a square root of its square which is evaluated as the difference between a sum of two positive terms and a third term. This is prone to loss of significance and one suspects that errors smaller than the square root of machine precision are not computable using the form \eqref{eq:error_es_quan}. This supposition is borne out in numerical results for RBM.

In this section, we provide two novel approaches to mitigate the deficiencies of classical RBM residual estimation. \annote{We begin first by discussing why the classical approach for evaluation of the error estimate via the formula \eqref{eq:error_es_quan} is not ideal when implemented in finite-precision arithmetic.
We follow this by presenting our robust residual-based error estimate, which \textit{directly} evaluates $\lVert r_N(\cdot; \bm{\mu})\rVert_{(X^{\mathcal{N}})'}$ without computing its squared value; this allows the new method to achieve errors much smaller than root machine precision. Finally, we present our residual-free approach that uses a surrogate for the residual that circumvents requirement of computing $\alpha_{LB}(\cdot)$. Thus the second approach is applicable to pPDE where no rigorous \textit{a posteriori} error estimate is available. Our {preliminary} analysis and numerical experiments suggests that the residual-free method is promising, but rigorous theoretical analysis demonstrating its utility in the RBM setting is the subject of ongoing work.}

\annote{
\subsection{Finite-precision limitations for residual norm evaluation}
The formula \eqref{eq:error_es_quan} is an expanded quadratic form for the expression
\begin{align}\label{eq:error-rN}
  \left\|r_N(\cdot;\bmu)\right\|_{(X^\calN)'}^2 = \left\| a(u_N^\calN - u^\calN, \cdot; \bmu) \right\|_{{(}{X^\calN}')}^2 = \left\| a(u_N^\calN, \cdot; \bmu) - f(\cdot; \bmu) \right\|_{(X^\calN)'}^2
\end{align}
At the root of the floating-point stagnation is evaluation of the above quadratic form via the formula
\[
(a - b)^2 = a^2 - 2 a b + b^2.
\]
Indeed, the right-hand side expansion of the above equation is \eqref{eq:error_es_quan}. A simple floating point error analysis reveals the following lemma, where we use $\epsilon$ to denote the machine precision, and $fl(\cdot)$ to denote the floating point representation of a number.
\begin{lemma}
\label{lemma:a_b}
In the case of $b = a + \calO(\epsilon)$, we have that
\[
fl(a^2 - 2 a b + b^2) = \calO(\epsilon){,}
\] 
while 
  \begin{align*}
    fl((a-b)^2) = \calO(\epsilon^2){.}
  \end{align*}
\end{lemma}
\begin{proof}
  {Assume $fl(x) = x(1+C \epsilon)$, where $C$ is an $\epsilon$-independent constant that does depend on $x$. We have $|C| \leq 1$, and in the sequel we use $C_j$ for $j = 1, 2 \ldots, $ to denote various such constants. Since $a-b = \calO(\epsilon)$, we have
\begin{align*}
fl(a^2 - 2 a b + b^2) &=  fl(a^2) - fl(2ab) + fl(b^2) + C_1 \epsilon \\
                                &= a^2(1+C_2 \epsilon) - 2ab(1+ C_3 \epsilon) + b^2(1 + C_4 \epsilon)  + C_1 \epsilon\\
                                &= (a - b)^2 + (C_1 + a^2C_2 - 2abC_3 + b^2C_4)\epsilon \\
                                & = (\calO(\epsilon))^2 + \calO(\epsilon) \\
                                & =  \calO(\epsilon).
\end{align*}
showing the first equation in the Lemma's conclusion. The second equation is straightforward since $fl((a-b)^2) = fl((\calO(\epsilon))^2) = (\calO(\epsilon))^2(1+C \epsilon) = (\calO(\epsilon))^2 = \calO(\epsilon^2)$. }
\end{proof}
The technical conclusion of this lemma can be communicated visually via Figure \ref{fig:a_b}. In this Figure, we introduce a parameter $\mu \in (0,1)$ to emulate the RBM setting, and set $b = a + \mu 4^{-N}$. We numerically compute the $\mu$-maximum of the expressions $\sqrt{(a-b)^2} = \mu 4^{-N}$ and $\sqrt{a^2 - 2 a b + b^2}$. 
\begin{figure}
\begin{center}
  \includegraphics[width=0.6\textwidth]{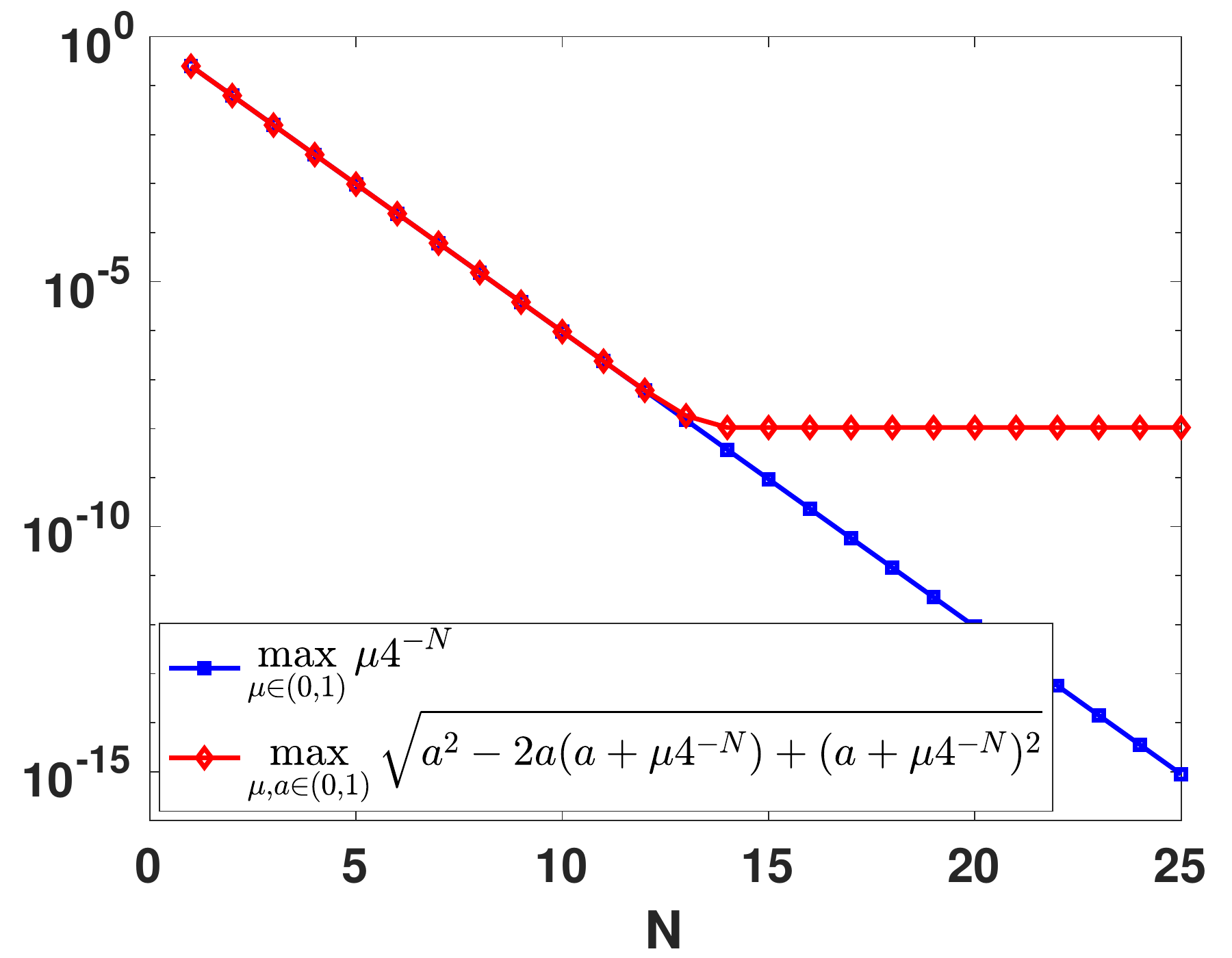}
\end{center}
  \caption{The loss of accuracy for the formula $(a - b)^2 = a^2 - 2 a b + b^2$, with $a$ randomly chosen from $(0, 1)$ and $b = a + \mu 4^{-N}$. Implemented in double precision, we have $\sqrt{\epsilon} \approx 10^{-8}$. We see that the $a^2 - 2 a b + b^2$ formula stagnates around this level.}
\label{fig:a_b}
\end{figure}
The figure results demonstrate that the expression $\sqrt{a^2 - 2 a b + b^2}$, representing the direct RBM estimate \eqref{eq:error_es_quan}, {stagnates} around $\sqrt{\epsilon}$, which is about $10^{-8}$ in IEEE double-precision. Direct usage of $\sqrt{(a-b)^2}$ does not suffer from this stagnation. This motivates the need for a robust approach to handle this case. One approach is described in \cite{Ohlberger2014}, and we describe a new method in the next two sections that also circumvents this issue.
}

\subsection{A new evaluation of the residual norm}
\label{sec:ee2}
\annote{
An intuitive explanation of our approach is as follows: Fix $u_N^\calN$ and consider the right-hand side expression in \eqref{eq:error-rN}. {We define $f\left(\cdot; \bmu\right) = f_{\parallel}\left(\cdot;\bmu\right) + f_{\perp}(\cdot;\bmu)$, where  $f_{\parallel}\left(\cdot;\bmu\right)$ is} a component of $f(\cdot;\bmu)$ that is parallel to $a(u_N^\calN,\cdot;\bmu)$ {and $f_{\perp}(\cdot;\bmu)$ is} a component that is perpendicular to $a(u_N^\calN,\cdot;\bmu)$. {Therefore, equation (\ref{eq:error-rN}) can be} rewrite as
\begin{align*}
  \left\|r_N(\cdot;\bmu)\right\|_{(X^\calN)'}^2 = \left\| \left[ a(u_N^\calN, \cdot; \bmu) - f_{\parallel}\left(\cdot;\bmu\right) \right] + f_{\perp}(\cdot;\bmu) \right\|_{(X^\calN)'}^2{.}
\end{align*}
Our improvement to the standard evaluation \eqref{eq:error_es_quan} computes the residual norm above by the following observation: Since the two separated terms under the norm are perpendicular, we can separate them via the Pythagorean theorem. We now present the details; writing the residual norm as a norm in $X^\calN$ instead of in the dual $(X^\calN)'$, we have
}
\begin{equation}
\label{eq:residualnorm}
\lVert r_N(\cdot; \bm{\mu})\rVert_{(X^{\mathcal{N}})'} = \left\lVert \sum_{q = 1}^{Q_f} \Theta_f^q(\bmu) \calC^q - \sum_{q = 1}^{Q_a} \sum_{m = 1}^N \Theta_a^q(\bmu) u^\calN_{Nm} \calL_m^q \right \rVert_{X^\calN}.
\end{equation}
Assume that $\calN > Q_a N$, which is a reasonable assumption in the RBM framework. We introduce the following subspace in the Hilbert space $X^\calN$:
\begin{align*}
  V_N &\coloneqq \mathrm{span} \left\{ \calL_1^1,\; \ldots,\; \calL_1^{Q_a},\; \ldots,\; \calL_N^1,\; \ldots,\; \calL_N^{Q_a}\right\}, \\
\end{align*}
Defining $V_N^\perp$ as the $X^\calN$-orthogonal complement of $V_N$, then let $P_N$ be the $X^\calN$-orthogonal projection onto $V_N$, and let $P_N^\perp$ be the orthogonal projection onto $V_N^\perp$.

We are now ready to state a preliminary result:
\begin{lemma}\label{lemma:residual}
The dual norm of the residual \eqref{eq:residualnorm} has the form
\begin{equation}
\label{eq:newresidualnorm}
\lVert r_N(\cdot; \bm{\mu})\rVert_{(X^{\mathcal{N}})'} = \sqrt{\left\lVert \sum_{q=1}^{Q_f} \Theta_f^q(\bmu) P_N^\perp \calC^q \right \rVert^2_{X^\calN} + \left\lVert \sum_{q=1}^{Q_f} \Theta_f^q(\bmu) P_N \calC^q - \sum_{q=1}^{Q_a} \sum_{m=1}^N \theta_a^q\left(\bmu\right) u_{N m}^\calN \calL_m^q \right \rVert^2_{X^\calN}}.
\end{equation}
\end{lemma}
\begin{proof}
  The result is a fairly straightforward computation in least-squares problems. We have 
  \begin{align*}
    r_N\left(\cdot,\bmu\right) &= \sum_{q = 1}^{Q_f} \Theta_f^q(\bmu) \calC^q - \sum_{q = 1}^{Q_a} \sum_{m = 1}^N \Theta_a^q(\bmu) u^\calN_{Nm} \calL_m^q \\
                               &= \sum_{q = 1}^{Q_f} \Theta_f^q(\bmu) P_N^\perp \calC^q + \left( \sum_{q = 1}^{Q_f} \Theta_f^q(\bmu) P_N \calC^q -  \sum_{q = 1}^{Q_a} \sum_{m = 1}^N \Theta_a^q(\bmu) u^\calN_{Nm} \calL_m^q\right).
  \end{align*}
  Note that the first term is an element of $P_N^\perp$ and the second term in parenthesis is an element of $P_N$. The conclusion follows from the Pythagorean {theorem}.
\end{proof}
This Lemma yields a computational procedure that ameliorates finite-precision loss of significance in numerical implementations.

\subsubsection{Implementation and offline-online decomposition}

In this section we treat elements of $X^\calN$ as Euclidean vectors in $\R^{\calN}$ and identify the norm $\|\cdot\|_{X^\calN}$ with the standard $\ell^2$ norm $\|\cdot\|$. When the norm $\|\cdot\|_{X^\calN}$ is different than the $\ell^2$ norm, the below discussion would proceed by inserting Gramian square root matrices in appropriate places so that the resulting weighted $\ell^2$ norm equals the norm on $X^\calN$.

The discretized version of $\calL_m^q \in X^\calN$, is $\vec{\calL}_m^q$, a $\calN \times 1$ vector. Similarly, we let $\vec{\calC^q} \in \R^{\calN}$ denote the vector representation of $\calC^q$. We use the $\vec{\calL}_m^q$ vectors to define a matrix $\mathcal{B}$ and its associated column-pivoted reduced QR factorization:
\begin{align*}
  \mathcal{B} &= \left( \vec{\calL}_1^1,\; \ldots,\; \vec{\calL}_1^{Q_a},\; \ldots,\; \vec{\calL}_N^1,\; \ldots,\;\vec{\calL}_N^{Q_a}\right) \in \R^{\calN \times Q_a N}{,} \\
  \mathcal{B} \mathcal{Z} &= \calQ \calR, \hskip 25pt \calQ \in \R^{\calN \times \mathrm{rank}(\mathcal{B})},\hskip 15pt  \calR \in \R^{\mathrm{rank}(\mathcal{B}) \times Q_a N},
\end{align*}
where $\mathcal{Z} \in \R^{Q_a N \times Q_a N}$ is a permutation matrix {obtained from the column-pivoted reduced QR factorization above}. Define {$\vec{c}\left(\bmu\right)$} the column vector
\begin{align*}
  {\vec{c}\left(\bmu\right)} &= \left( \Theta_a^1(\bmu) u^\calN_{N1},\; \ldots,\; \Theta_a^{Q_a}(\bmu) u^\calN_{N1},\; \ldots,\; \Theta_a^1(\bmu) u^\calN_{NN},\; \ldots \Theta_a^{Q_a}(\bmu) u^\calN_{NN} \right)^T \in \R^{Q_a N}.
\end{align*}
Finally, define $V_N \subset X^\calN$ as the column space of $\calQ$, along with some associated projection matrices:
\begin{align*}
  V_N &\coloneqq \mathrm{range}\left(\calQ\right), & \mathcal{P}_N &= \calQ \calQ^T, & \mathcal{P}_N^\perp &= I - \mathcal{P}_N \eqqcolon \mathcal{W} \mathcal{W}^T.
\end{align*}
\annote{We remark that we allow the column space of $\calQ$ to be less than the number of columns of $\mathcal{B}$ through, e.g. a rank-revealing QR factorization. In cases where this rank deficiency is utilized, this makes our algorithm more online-efficient than that of \cite{Ohlberger2014}, where a Gram-Schmidt step with re-iteration is used on the full dimension of the column space.} 
With $V_N^\perp$ as the $\R^{\calN}$-orthogonal complement of $V_N$, then $\mathcal{P}_N$ orthogonally projects onto $V_N$, and $\mathcal{P}_N^\perp$ orthogonally projects onto $V_N^\perp$. The columns of $\mathcal{W} \in \R^{\calN \times \left(Q_a N - \mathrm{rank}\left(\mathcal{B}\right)\right)}$ are formed from any orthonormal basis for $V_N^\perp$. We can now state conclusion of Lemma \ref{lemma:residual} in terms of vectors and matrices:
\begin{theorem}
The dual norm of the residual \eqref{eq:residualnorm} can be evaluated by
\begin{equation}
\label{eq:newresidualcomputation}
\lVert r_N(\cdot; \bm{\mu})\rVert^2_{(X^{\mathcal{N}})'} = \left\lVert \sum_{q=1}^{Q_f} \Theta_f^q(\bmu) \mathcal{W}^T \vec{\calC}^q \right \rVert^2 + 
\left\lVert \sum_{q=1}^{Q_f} \Theta_f^q(\bmu) \calQ^T \vec{\calC}^q - \calR \mathcal{Z}^T {\vec{c}\left(\bmu\right)} \right \rVert^2.
\end{equation}
\end{theorem}
\begin{proof}
The second term in this result equals the second term in \eqref{eq:newresidualnorm} due to the fact that $\calQ^T \calQ = I$ and that $\calQ^T$ is an isometric map on the range of $\calQ$:
\begin{align*}
  \left\lVert \sum_{q=1}^{Q_f} \Theta_f^q(\bmu) {\mathcal{P}_N} \calC^q - \sum_{q=1}^{Q_a} \sum_{m=1}^N \theta_a^q\left(\bmu\right) u_{N m}^\calN \calL_m^q \right \rVert^2_{X^\calN} &= \left\lVert \sum_{q=1}^{Q_f} \Theta_f^q(\bmu) \mathcal{P}_N \vec{\calC}^q - \mathcal{B} {\vec{c}\left(\bmu\right)} \right \rVert^2 \\
                                                                                                                                                                          &= \left\lVert \sum_{q=1}^{Q_f} \Theta_f^q(\bmu) \mathcal{Q} \mathcal{Q}^T \vec{\calC}^q - \mathcal{Q} \mathcal{R} \mathcal{Z}^T {\vec{c}\left(\bmu\right)} \right \rVert^2 \\
                                                                                                                                                                          &= \left\lVert \sum_{q=1}^{Q_f} \Theta_f^q(\bmu) \mathcal{Q}^T \vec{\calC}^q - \mathcal{R} \mathcal{Z}^T {\vec{c}\left(\bmu\right)} \right \rVert^2
\end{align*}
That the first term in the result equals the first term in \eqref{eq:newresidualnorm} is the result of a similar computation but using $\mathcal{W}$ instead of $\mathcal{Q}$:
\begin{align*}
  \left\lVert \sum_{q=1}^{Q_f} \Theta_f^q(\bmu) \mathcal{P}_N^\perp \vec{\calC}^q \right \rVert^2 &= \left\lVert \mathcal{W} \sum_{q=1}^{Q_f} \Theta_f^q(\bmu) \mathcal{W}^T \vec{\calC}^q \right \rVert^2 \\
                                                                                                  &= \left\lVert \sum_{q=1}^{Q_f} \Theta_f^q(\bmu) \mathcal{W}^T \vec{\calC}^q \right \rVert^2{.}
\end{align*}
\end{proof}
The theorem above immediately reveals an offline-online decomposition: The $\bmu$-independent parts of the formula involve the matrix products 
\begin{align*}
  \mathcal{W}^T \vec{\calC}^q, \hskip 25pt \mathcal{Q}^T \vec{\calC}^q, \hskip 25pt \mathcal{R}\mathcal{Z}^T,
\end{align*}
none of which have any dimensions dependent on $\calN$, but are simply dependent on $N$ and $Q_a$. Thus, these matrices may be precomputed and stored independent of $\bmu$. One $\bmu$-dependent component involves the affine coefficients $\Theta_f^q(\bmu)$, which are scalar and thus easy to compute. Finally, the coefficients ${\vec{c}\left(\bmu\right)}$ can be computed explicitly via the affine coefficients $\Theta_a^q(\bmu)$ and from the RBM approximation $u^\calN_N$ in \eqref{eq:rbsolution}. Since computing RBM approximation is $\calN$-independent, the entire residual norm computation via \eqref{eq:newresidualcomputation} is $\calN$-independent in the online phase. 
We denote $\Delta_N(\bmu)$ in \eqref{eq:error_estimator_en} with residual norm computed via this new approach by ${\mathcal E}_2(\bmu; n)$.

\subsection{A residual-free error indicator}
\label{sec:ee3}
Much of the RBM algorithm is dependent on rigor of the inequality in \eqref{eq:error_estimator_en}. However, in practical situations one may not have access to such a computable \textit{a posteriori} error estimator. This happens, for instance, with sufficiently complicated nonlinear pPDE's for which mathematical analysis is difficult or infeasible. In other situations, a rigorous error estimate like \eqref{eq:error_estimator_en} may exist, but is not {easily} computable in a $\calN$-independent fashion due to nonlinearity of the pPDE. In this case, the offline parameter selection portion of the RBM algorithm may be so expensive as to outweigh any computational saving gained during the online phase.

In either of the cases above, one still hopes to use an efficient model-order reduction strategy like RBM, but with an understanding that mathematically rigorous error certification may be lost.

A strategy for devising an error estimate for such a case is the subject of this section. While presumably one always has access to the pPDE residual, one could then use $\|r_N(\cdot; \bmu)\|_{\calN}$ as an error estimator that is not mathematically rigorous. However, this estimator may also not be $\calN$-independent, so that such an estimator is both non-rigorous and expensive. 

The alternative we suggest is as follows: recall expression \eqref{eq:rbsolution} that expresses the RBM approximation $u_N^\calN(\bmu)$ in terms of the snapshots {$u^\calN(\bmu^n)$}, $n=1, \ldots, N$, via the expansion coefficients $\left\{ u_{Nm}^\calN(\bmu) \right\}_{m=1}^N$. As indicated, these expansion coefficients are functions of $\bmu$, and {play role of basis functions. Furthermore,} they satisfy
\begin{align}\label{eq:kronecker}
  {u^\calN_{N m}\left(\bmu^n\right)} = \delta_{n,m},
\end{align}
where $\delta_{n,m}$ is the Kronecker delta. The above property is a direct consequence of condition \eqref{eq:reduced_system} that defines the RBM solution. Therefore, the coefficient functions {$u^\calN_{N m}\left(\cdot\right)$} are actually cardinal Lagrange interpolants associated to the space of functions defined by their span. Note that, given any $\bmu$, the cost of evaluation of these cardinal functions does \textit{not} depend on $\calN$, since these coefficients are computed from the RBM solution.

We now rephrase the essential portion of the offline RBM phase: given the the current parameter values $\bmu^1, \ldots, \bmu^N$ along with the current subspace of parameter-dependent functions 
\begin{align*}
  \mathrm{span} \left\{ { u^\calN_{N1}(\cdot), \ldots, u^\calN_{N N}(\cdot)} \right\},
\end{align*}
can we compute the next parameter value $\bmu^{N+1}$? Abstractly, this can be interpreted as an interpolation problem to find a nested sequence of interpolation points $\bmu^j, \bmu^{j+1}, \ldots$, associated to a nested sequence of function spaces. 

We consider one potential solution to this problem, inspired by concepts in polynomial approximation. We take the next point $\bmu^{N+1}$ as the point that maximizes a function of these cardinal interpolants:
\begin{align}\label{eq:ee3}
  \widetilde{\Delta}_N(\bmu) &= \left(\sum_{m=1}^N \left| {u^\calN_{N m}(\bmu)} \right|\right){.}
\end{align}
The above function is simply the Lebesgue function from interpolation theory {(i.e., the norm of an interpolation operator).}. Note that evaluation of this function depends only on the RBM solution, and does not directly involve computation of residual norms, nor does it require mathematically rigorous \textit{a posteriori} error estimates. We show in Lemma \ref{lemma:lebesgue} below that $\widetilde{\Delta}_N$ does indeed match the behavior of $e_N(\bmu)$; it is therefore quite useful in selecting RBM parameter values to compute snapshots. However, the relationship between $e_N$ and $\widetilde{\Delta}_N$ involves a multiplicative scaling constant that is in general an uncomputable best approximation error. Since we cannot compute this scaling constant, we cannot certify the error committed by parameter values picked with this method. However, we show in our numerical results section that choosing parameter samples via greedy maximization of \eqref{eq:ee3} empirically produces results comparable to using the RBM error estimate \eqref{eq:error_estimator_en}. {Our ongoing work seeks to ``certify'' the surrogate error generated via this approach. A simple approach for approximating a lower bound for $\epsilon_N$ can be $\|u^\calN(\bmu^{N+1}) - u_N^{\calN}(\bmu^{N+1})\|_{{X}}$, which is a computable quantity requiring no additional PDE simulations in the RBM context. Naturally, the effectiveness of this and related approaches will need to be carefully studied in the future.}

We emphasize that this procedure is similar to, but distinct from, empirical interpolation procedures \cite{BarraultMaday2004}. In empirical interpolation, one essentially has an $(N+1)$-dimensional space with $N$ points, and uses the discrepancy between the $(N+1)$-dimensional space and the $N$ points to pick the $(N+1)$st point. We cannot do this here since the $(N+1)$-dimensional space depends explicitly on the sought point $\bmu^{N+1}$. Thus, our strategy using the objective \eqref{eq:ee3} circumvents the needs for identification of the higher-dimensional space.

\subsubsection{Characterization of $\widetilde{\Delta}_N$}

We now motivate the choice of $\widetilde{\Delta}_N(\bmu)$ as an error indicator. 
We introduce the space of $X^\calN$-valued functions in $L^\infty(\mathcal{D})$:
\begin{align*}
  L^\infty\left(\mathcal{D}, X^\calN \right) &= \left\{ u: \mathcal{D} \rightarrow X^\calN \;\; \big| \;\; \| u\|_{L^\infty\left(\mathcal{D}, X^\calN\right)} < \infty \right\}, & \| u\|_{L^\infty\left(\mathcal{D}, X^\calN\right)} &\coloneqq \sup_{\bmu \in \mathcal{D}} \| u(\bmu)\|_{X^\calN}{.}
\end{align*}
A {subspace} of particular interest is those functions in $L^\infty\left(\mathcal{D}, X^\calN \right)$ whose $\bmu$-variation is prescribed by the cardinal functions {$u^\calN_{N m}$}:
\begin{align*}
  U_N = \left\{ u = \sum_{{m}=1}^N y_m u^\calN_{N m}(\bmu) \;\; \big| \;\; y_1, \ldots, y_N \in X^\calN \right\} \subset L^\infty\left( \mathcal{D}, X^\calN \right){.}
\end{align*}
Note that the RBM solution is in $U_N$, and the truth solution is in $L^\infty\left(\mathcal{D}, X^\calN\right)$:
\begin{align*}
  u_N^\calN &\in U_N, & u^\calN \in L^\infty\left(\mathcal{D}, X^\calN\right).
\end{align*}
The former is true by inspection of \eqref{eq:rbsolution}, and the latter is true because of the truth discretization versions of the uniform ellipticity and continuity assumptions \eqref{eq:con} -- \eqref{eq:coe_rev}. The following result then applies.
\begin{lemma}\label{lemma:lebesgue}
  Let ${\left\|e_N(\bmu)\right\|_{X^\calN}} = \left\| u_N^\calN(\bmu) - u^\calN(\bmu) \right\|_{X^\calN}$ be the RBM error committed at parameter value $\bmu$. Then
  \begin{align*}
   {\left\|e_N(\bmu)\right\|_{X^\calN}}  \leq \left(1 + \widetilde{\Delta}_N(\bmu)\right) \epsilon_N(u^\calN),
  \end{align*}
  where $\epsilon_N$ is the $\bmu$-independent quantity,
  \begin{align*}
    \epsilon_N(u^\calN) \coloneqq \inf_{v \in U_N} \left\| u^\calN - v \right\|_{L^\infty\left(\mathcal{D}, X^\calN\right)}{.}
  \end{align*}
\end{lemma}
\begin{proof}
  The result is an exercise in a pointwise version of Lebesgue's Lemma for projective approximations. We first define two operators. The first, $P_N : L^\infty\left(\mathcal{D}, X^\calN\right) \rightarrow U_N$, is the interpolative projection operator defined by
  \begin{align*}
    P_N v &= \sum_{m=1}^N v\left(\bmu^{{m}}\right) {u^\calN_{N m}(\bmu)}, & v &\in L^\infty\left(\mathcal{D}, X^\calN\right).
  \end{align*}
  The second operator is $\delta_{\bmu} : L^\infty\left(\mathcal{D}, X^\calN\right) \rightarrow X^\calN$, corresponding to point-evaluation at $\bmu \in \mathcal{D}$:
  \begin{align*}
    \delta_{\bmu} v &= v(\bmu), & v &\in L^\infty\left(\mathcal{D}, X^\calN\right).
  \end{align*}
  Note that
  \begin{align*}
    P_N u^\calN{(\bmu)} = \sum_{m=1}^N u^\calN(\bmu^m) {u^\calN_{N m}(\bmu)} = u^\calN_N{(\bmu)}.
  \end{align*}
  Now let $v$ be any element in $U_N$, {so that $P_N v = v$.}
  Then
  \begin{align}
    \nonumber {\left\|e_N(\bmu)\right\|_{X^\calN}}  &= \left\| u_N^\calN(\bmu) - u^\calN(\bmu)\right\|_{X^\calN} \leq \left\| u_N^\calN(\bmu) - v(\bmu) \right\|_{X^\calN} + \left\| v(\bmu) - u^\calN(\bmu)\right\|_{X^\calN} \\
              \nonumber &{=} \left\| {P_N \left[ u^\calN - v \right](\bmu)} \right\|_{X^\calN} + \left\| u^\calN(\bmu) - v(\bmu) \right\|_{X^\calN} \\
              \nonumber &= \left\| \delta_{\bmu} P_N \left[ u^\calN - v \right] \right\|_{X^\calN} + \left\| u^\calN(\bmu) - v(\bmu) \right\|_{X^\calN} \\
    \label{eq:lemma-1}&\leq \left[ 1 + \left\| \delta_{\bmu} P_N \right\| \right] \left\| u^\calN - v \right\|_{L^\infty\left(\mathcal{D}, X^\calN\right)},
  \end{align}
  where $\|\delta_{\bmu} P_N\|$ is the induced operator norm. We can directly compute
  \begin{align*}
    \left\| \delta_{\bmu} P_N \right\| &= \sup_{\|w\|_{L^\infty(\mathcal{D}, X^\calN)} = 1} \left\| \delta_{\bmu} P_N w \right\|_{X^\calN} \\
                                       &= \sup_{\|w\|_{L^\infty(\mathcal{D}, X^\calN)} = 1} \left\| \sum_{m=1}^N {u^\calN_{N m}(\bmu)} w(\bmu^m) \right\|_{X^\calN} \\
                                       &\leq \sum_{m=1}^N \left| {u^\calN_{N m}(\bmu)}\right| \sup_{\|w\|_{L^\infty(\mathcal{D}, X^\calN)} = 1} \left\| w(\bmu^m) \right\|_{X^\calN} \\
                                       &\leq \sum_{m=1}^N \left| {u^\calN_{N m}(\bmu)}\right| = \widetilde{\Delta}_N(\bmu){.}
  \end{align*}
  Using the above in \eqref{eq:lemma-1} and infimizing over $v$ yields
  \begin{align*}
    {\left\|e_N(\bmu)\right\|_{X^\calN}}  \leq \left[ 1 + \widetilde{\Delta}(\bmu) \right] \inf_{v \in U_N} \left\| v - u^\calN\right\|_{L^\infty\left(\mathcal{D}, X^\calN\right)} = \left[ 1 + \widetilde{\Delta}_N(\bmu) \right] \epsilon_N(u^\calN){.}
  \end{align*}
\end{proof}

Lemma \ref{lemma:lebesgue} states that, relative to the best approximant from $U_N$, the RBM solution $u_N^\calN$ commits a $\bmu$-pointwise error that scales monotonically with $\widetilde{\Delta}_N(\bmu)$. This is, essentially, a generalization of Lebesgue's Lemma in approximation theory.

Therefore, at any iteration $n$, the function $\widetilde{\Delta}_n(\bmu)$ gives a qualitative indication of the error at $\bmu$, and so choosing a new snapshot parameter $\bmu^{n+1}$ at the maximum of this function is a greedy function that indirectly seeks to minimize the RBM error. However, it does not provide a certifiable error since we cannot compute the value of $\epsilon_N$. \annote{Furthermore, our Lemma shows only that our residual-free error estimate is an upper bound for the true error; a rigorous procedure would also establish that our error estimate is a lower bound. Without this lower bound guarantee, one can contrive situations where our error estimate chooses parameter values that are not indicative of RBM subspace quality. However, we have not observed this in any examples we have tried. Our analysis cannot currently exclude the possibility of such pathological problems, and our ongoing work seeks to establish a lower bound estimate that would rigorously justify our residual-free objective.}

\section{Numerical results}
\label{sec:results}
In this section, we present numerical examples to demonstrate the accuracy and efficiency of the proposed two new 
approaches. The difference between all these approaches appears only in how parameter values $\bmu^j$ are selected for use in the RBM algorithm.  We have three approaches to compare: (i) The standard RBM strategy that uses $\Delta_N(\bmu)$ for the greedy objective as defined in \eqref{eq:error_estimator_en} with the residual norm computed using \eqref{eq:error_es_quan}. (ii) The RBM strategy that again uses $\Delta_N(\bmu)$ from \eqref{eq:error_estimator_en} as the greedy objective, but now uses the formula \eqref{eq:newresidualnorm} to compute the residual norm. (iii) The greedy objective is {$\widetilde{\Delta}_{N}(\bmu)$}, as defined in \eqref{eq:ee3}.

Since the parameter values selected by each of these three procedures is different, we use the subscript $i = 1,2 , 3$, to differentiate quantities for these methods. 
{We denote these error bounds are $\mathcal{E}_{i}(\bmu)$ for $i = 1, 2, 3$. More specifically,  
\begin{alignat}{2}
\mathcal{E}_{1}(\bmu) \quad = & \quad \Delta_{N}(\bm{\mu})  \text{ defined in (\ref{eq:error_estimator_en})}\\
\vspace{0.1in}
\mathcal{E}_{2}(\bmu) \quad = & \quad \frac{\lVert r_N(\cdot; \bm{\mu})\rVert_{(X^{\mathcal{N}})'}}{\alpha^{\mathcal{N}}_{LB}(\bm{\mu})}, \text{ where } \lVert r_N(\cdot; \bm{\mu})\rVert_{(X^{\mathcal{N}})'} \text{ is computed in } \eqref{eq:newresidualnorm}\\
\vspace{0.1in}
\mathcal{E}_{3}(\bmu) \quad = & \quad \widetilde{\Delta}_{N}(\bmu)
\end{alignat}
}

{Note that the} first two approaches have rigorous error certification values. 
We use $u^{\mathcal N}_{N, {\mathcal E}_i}(\mu)$  to denote the RBM approximation using greedy strategy $i$, for $i = 1,2 ,3$.

We test the three implementations on three problems, and present the results respectively in each subsection below.

\subsection{Two cases with 1-dimensional parameter}

We first test the three RB methods on the following equations with one parameter.
\begin{subequations}
\begin{alignat}{1}
\label{eq:1Ddiff_prob}
&(1+\mu x)u_{xx} + u_{yy} = e^{4xy} \quad {\rm on} \quad \Omega.\\
\label{eq:1Ddiff_prob_disc}
&(1+\ell(\mu) x)u_{xx} + u_{yy} = e^{4xy} \quad {\rm on} \quad \Omega.
\end{alignat}
\end{subequations}
The first equation has continuous dependence on the parameter while the second has discontinuous dependence by having
\[
  \ell(\mu) = \sin \left((\mu - \mbox{sign}(\mu)) \frac{\pi}{2} \right ), \quad \mu \in \mathcal{D}{.}
\]
We take the physical domain as $\Omega = [-1, 1]\times[-1, 1]$ and impose homogeneous Dirichlet boundary conditions on $\partial \Omega$. The truth approximation is a spectral Chebyshev collocation method based on $50$ degrees of freedom for each direction. The parameter domain $\mathcal{D}$ for $\mu$ is taken to be $[-0.995, 0.995]$, and the training set $\Xi_{\rm train}$ a uniform Cartesian grid with 512 equally spaced points. \\

We show in Figure \ref{fig:1d_conv} (left) the history of convergence for the three approaches. The classical approach stagnates before reaching the square root of machine epsilon as expected. However, both new approaches have worst-case error estimate and corresponding exact error converging further toward machine accuracy. It is worth noting that the residual-free function $\widetilde{\Delta}_N(\bmu)$ used in the greedy scheme for method $i=3$ is not a rigorous error bound. However, the RB space built from its maximizers 
has similar approximation properties, as confirmed by the cyan curve in Figure \ref{fig:1d_conv} (left).
The fact that $\widetilde{\Delta}_N$ captures the pattern of the true error as the parameter $\mu$ varies is shown in Figure \ref{fig:1d_conv} (right). Here the black $\widetilde{\Delta}_{10}(\mu)$ curve is multiplied by $10^{-4}$ to achieve better alignment with the error curve. 
Finally, we show in Figure \ref{fig:1d_lags} the $10$ Lagrange shape functions {$u^\calN_{N m}(\cdot)$ as implicitly defined in \eqref{eq:rbsolution} and \eqref{eq:reduced_system}} when a $10$-dimensional RB space is used. 

The Lagrange shape functions used in the residual-free method inherit the structure dictated by the PDE. For example, they are discontinuous if the PDE enforces this. The second example \eqref{eq:1Ddiff_prob_disc} has discontinuous parameter dependence, and we show the associated Lagrange shape functions for this case in Figure \ref{fig:1d_discontinuity}. The results show both that convergence is not directly affected by the discontinuity since the Lagrange functions now inherit discontinuous dependence from the PDE.

\begin{figure}
\begin{center}
  \includegraphics[width=0.49\textwidth]{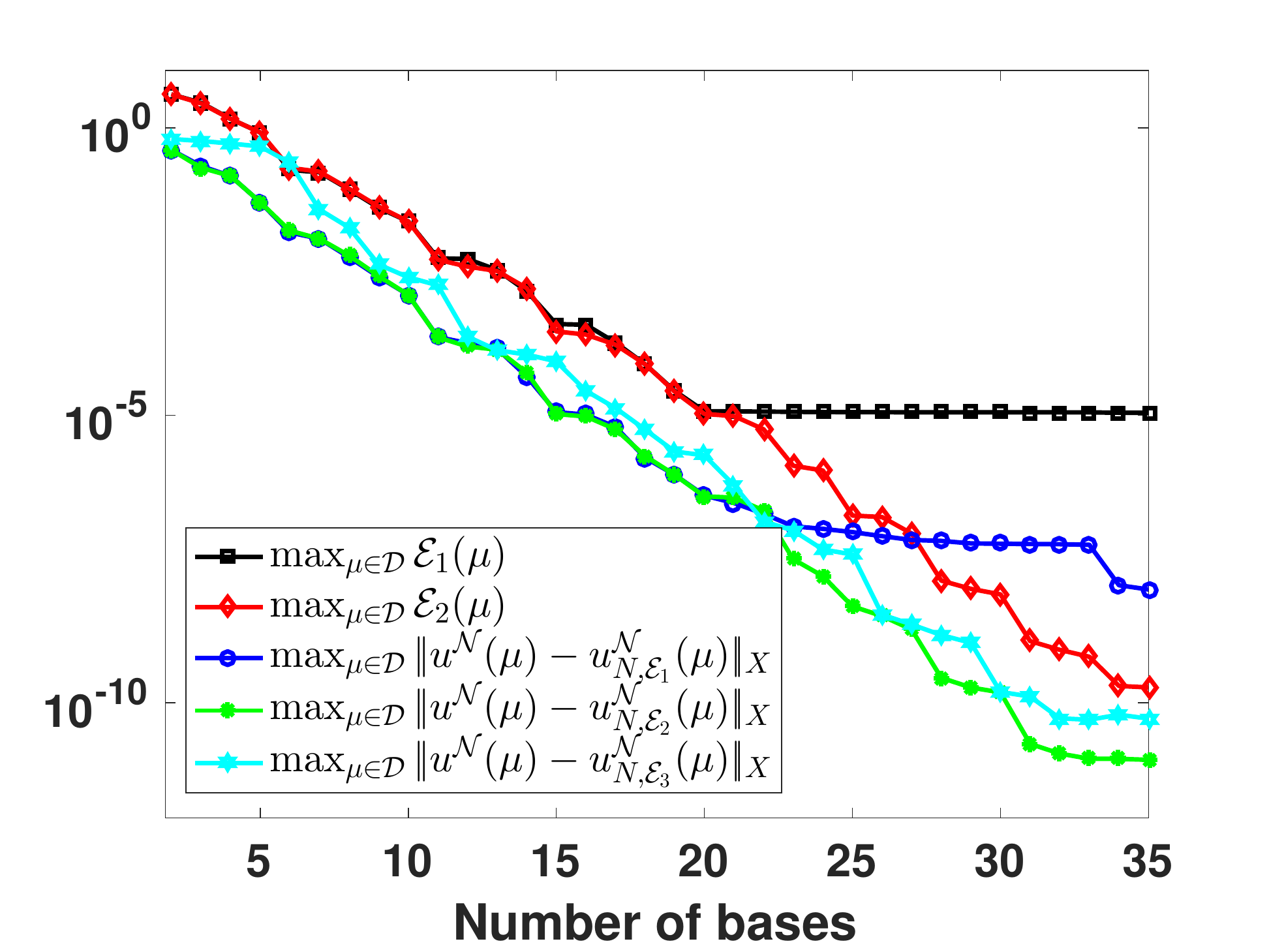}
  \includegraphics[width=0.49\textwidth]{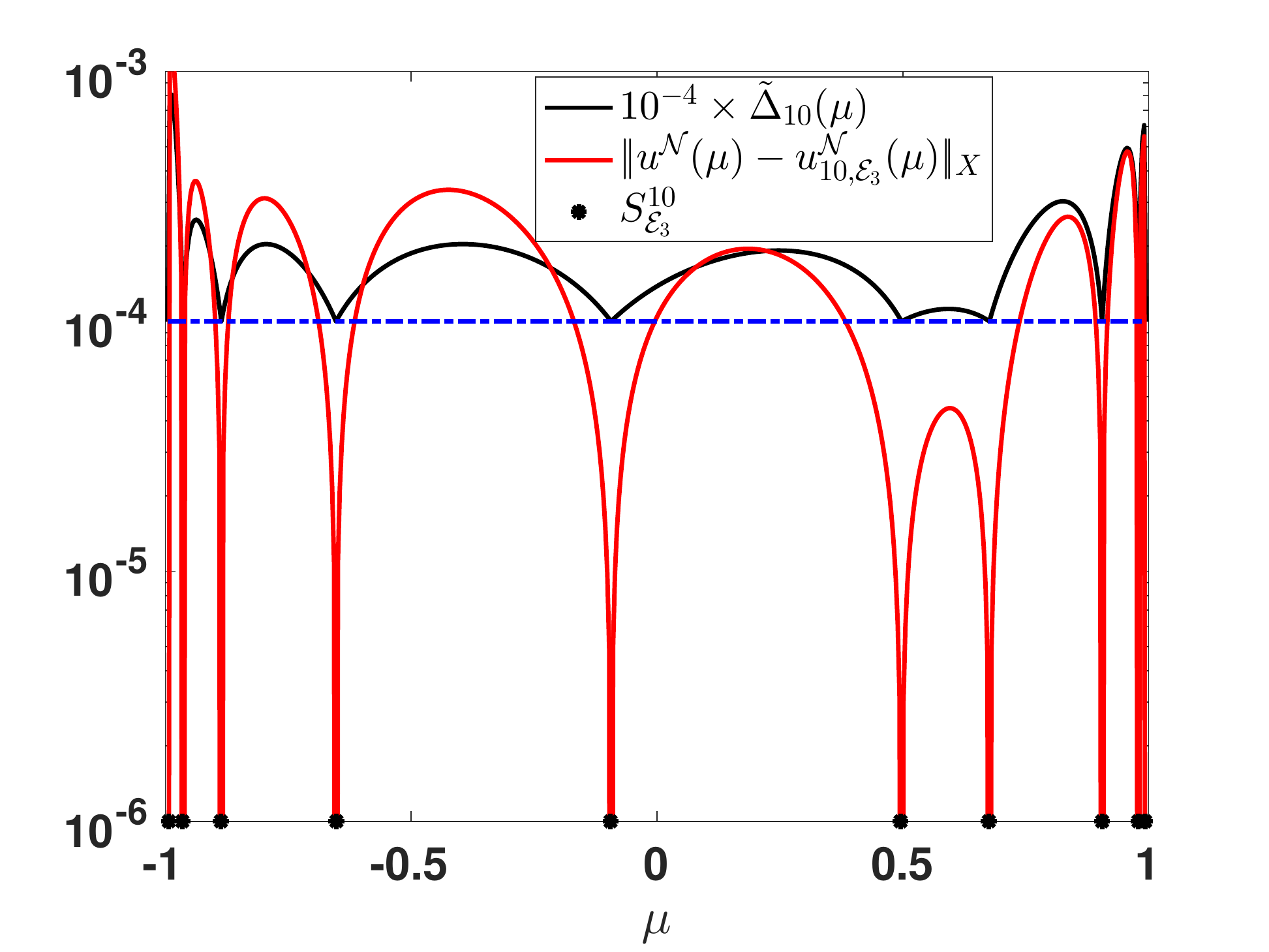}
\end{center}
\caption{Results for one-parameter case \eqref{eq:1Ddiff_prob}: The comparison of the three approaches (Left) and a demonstration that the residual-free error indicator matches with the true error well (Right).}
\label{fig:1d_conv}
\end{figure}

\begin{figure}
\begin{center}
  \includegraphics[width=\textwidth]{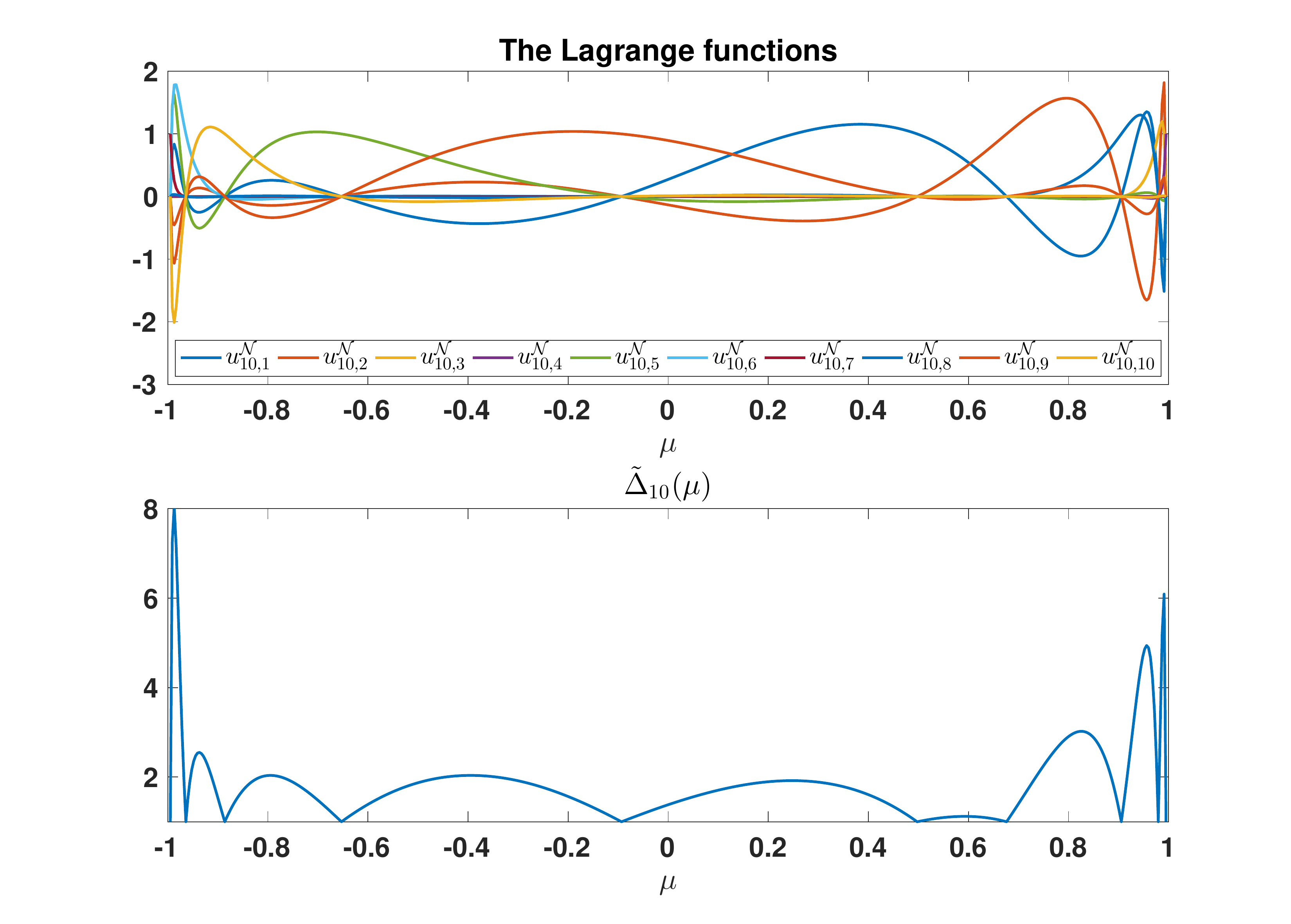}
\end{center}
\caption{The Lagrange shape functions.}
\label{fig:1d_lags}
\end{figure}

\begin{figure}
\begin{center}
  \includegraphics[width=0.49\textwidth]{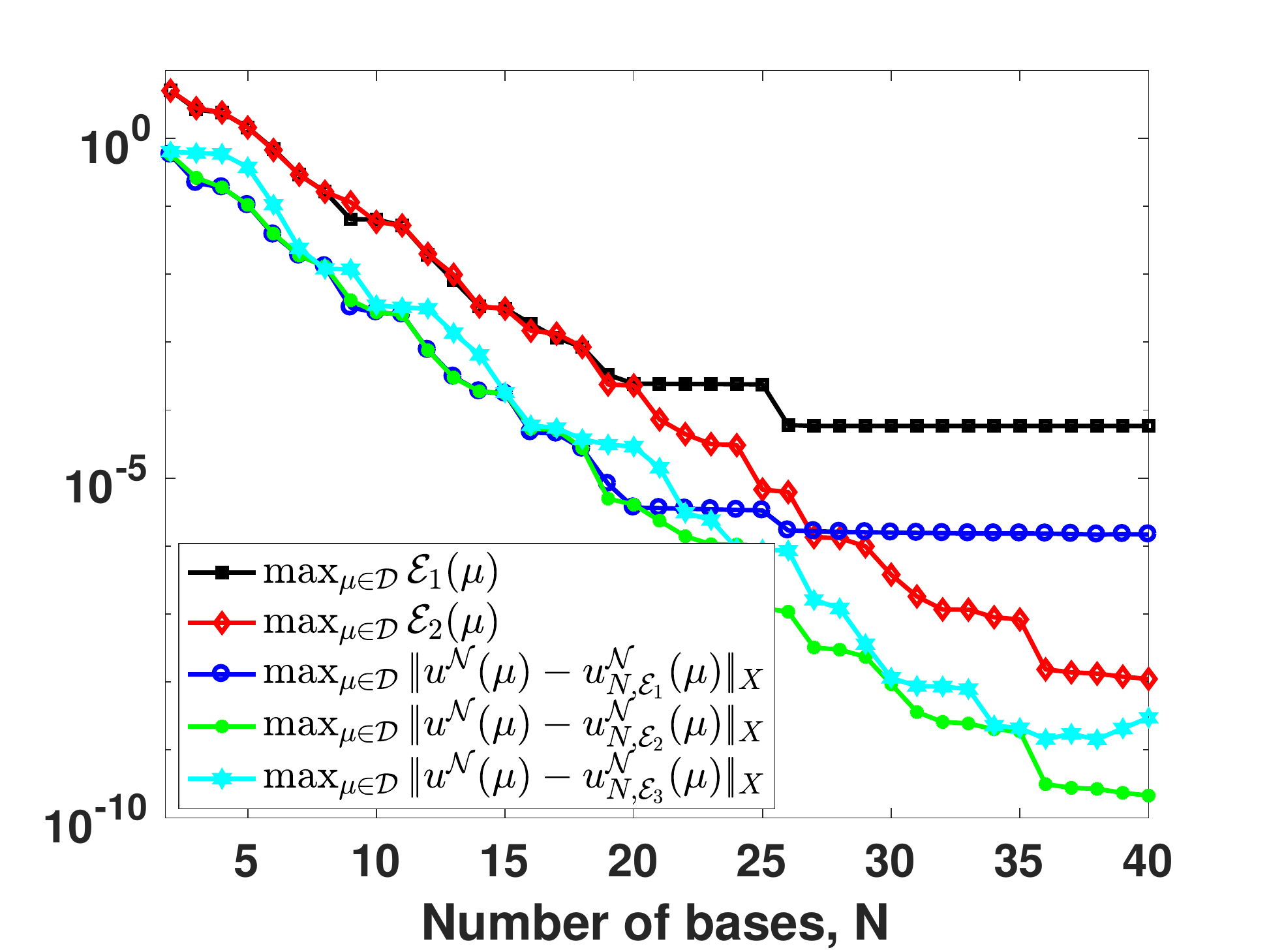}
  \includegraphics[width=0.49\textwidth]{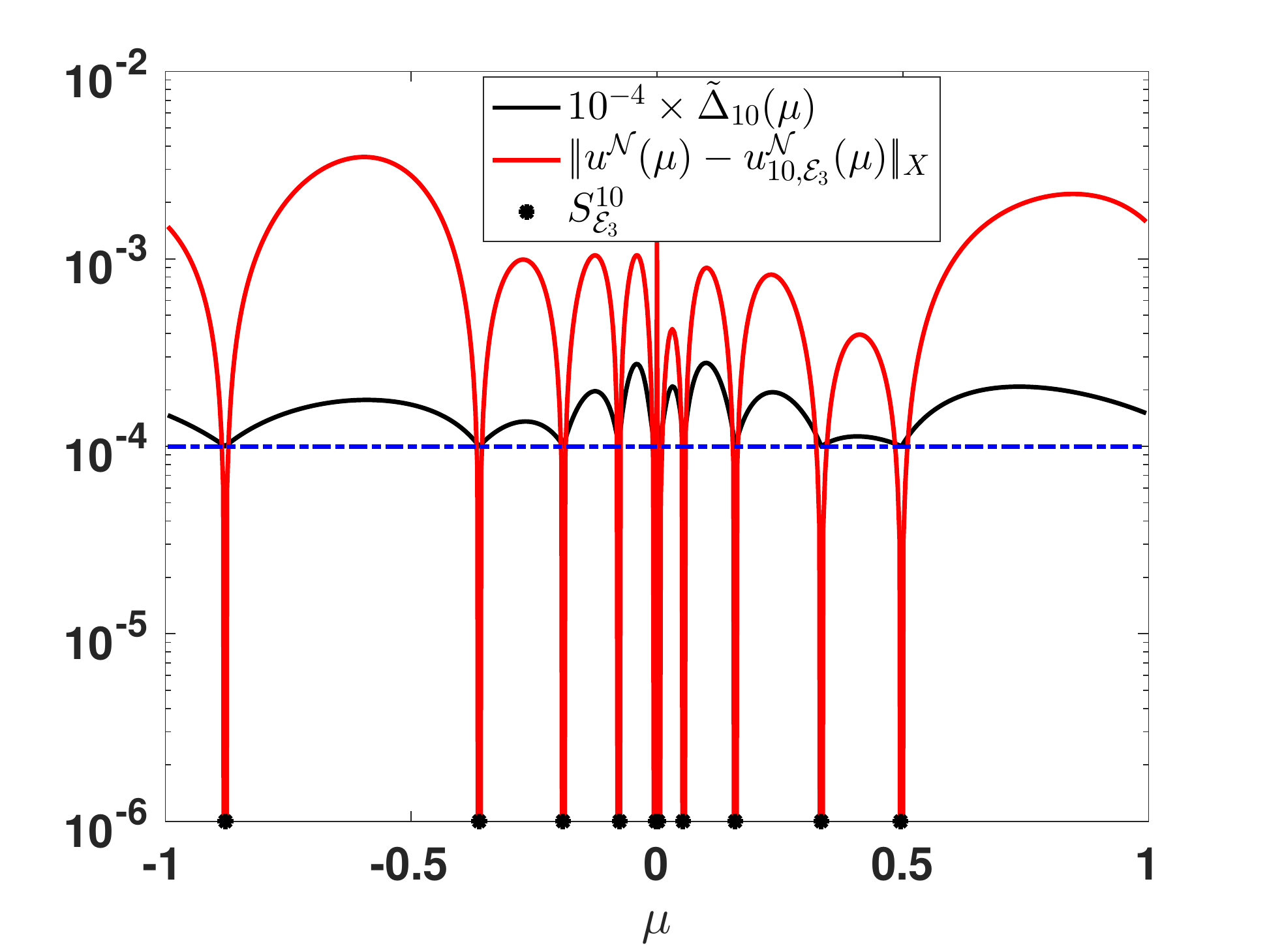}
  \includegraphics[width=\textwidth]{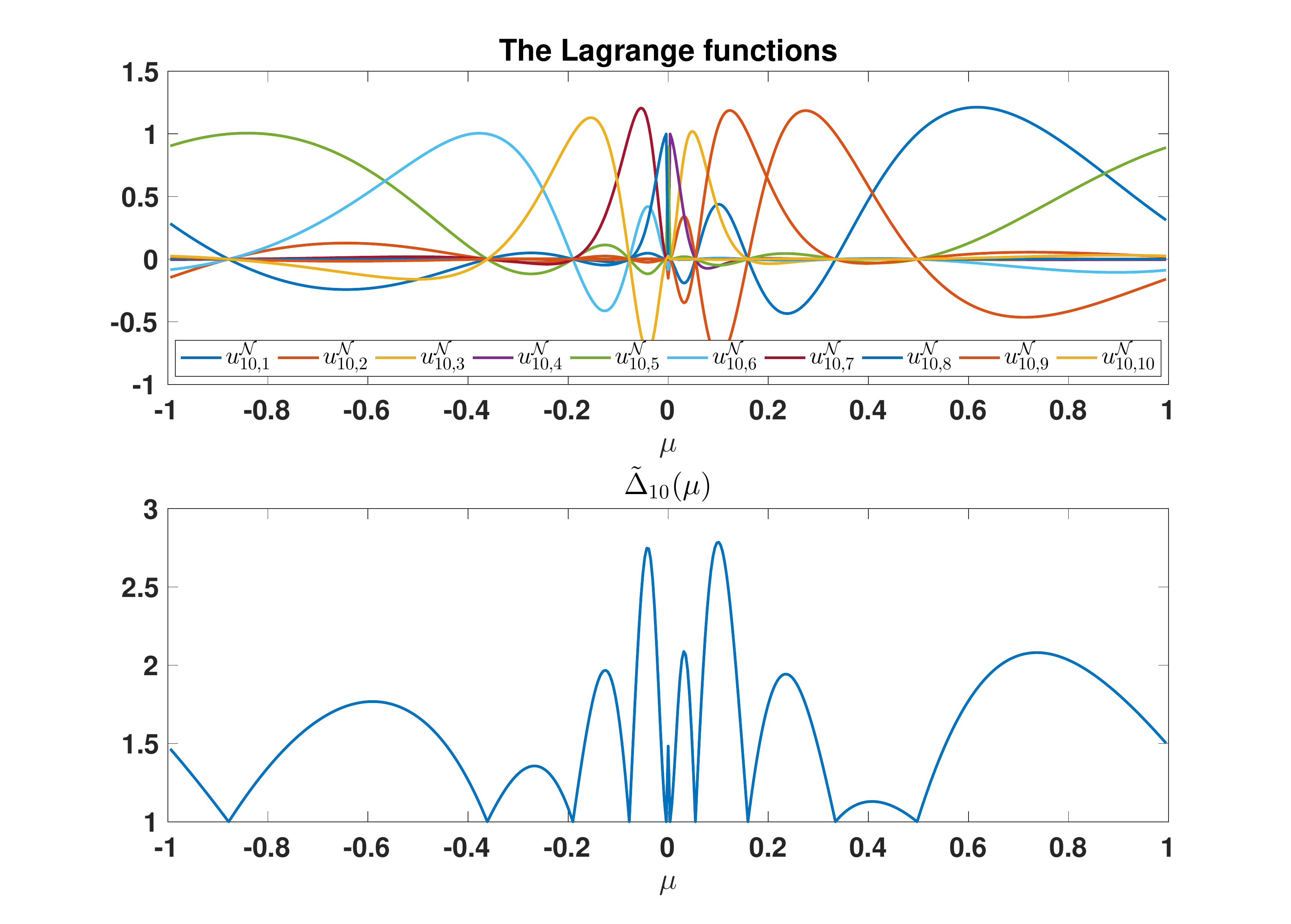}
\end{center}
\caption{Results for the one-dimensional case with discontinuous dependence on parameter \eqref{eq:1Ddiff_prob_disc}. 
  Top left us the comparison of histories of convergence for the three approaches. Top right demonstrates that the residual-free error indicator {roughly} matches with the true error well even when the parametric dependence is discontinuous. {Note that while error minima locations are successfully predicted, the maximum location(s) may be different.} The bottom figure shows discontinuous Lagrange functions that are necessary for the RB solution to approximate the truth solution well.}
\label{fig:1d_discontinuity}
\end{figure}

\subsection{The first 2-dimensional test case}

As a first test case with $2$-dimensional parameter, we consider the following equation whose solution  space ends up being well approximated by a $40$-dimensional RBM surrogate.
\begin{equation}
\label{eq:2Ddiff_prob1}
-u_{xx} - \mu_1 u_{yy} - \mu_2 u =  -10\sin(8 x (y-1)) \quad {\rm on} \quad \Omega.
\end{equation}
The physical domain is $\Omega = [-1 ,1] \times [-1,1]$ and we impose homogeneous Dirichlet boundary conditions on $\partial \Omega$. The truth approximation is a spectral Chebyshev collocation method with $\mathcal{N}_{x} = 50$ degrees of freedom in each direction. This means the truth approximation has dimension $\mathcal{N} = \mathcal{N}^2_{x}$. The parameter domain $\mathcal{D}$ for $(\mu_{1}, \mu_{2})$  is taken to be $[0.1, 4] \times [0, 2]$. For the training set $\Xi_{\rm train}$ we discretize $\calD$ using a tensorial $129 \times 65$ uniform Cartesian grid. \\

\begin{figure}
\begin{center}
  \includegraphics[width=0.49\textwidth]{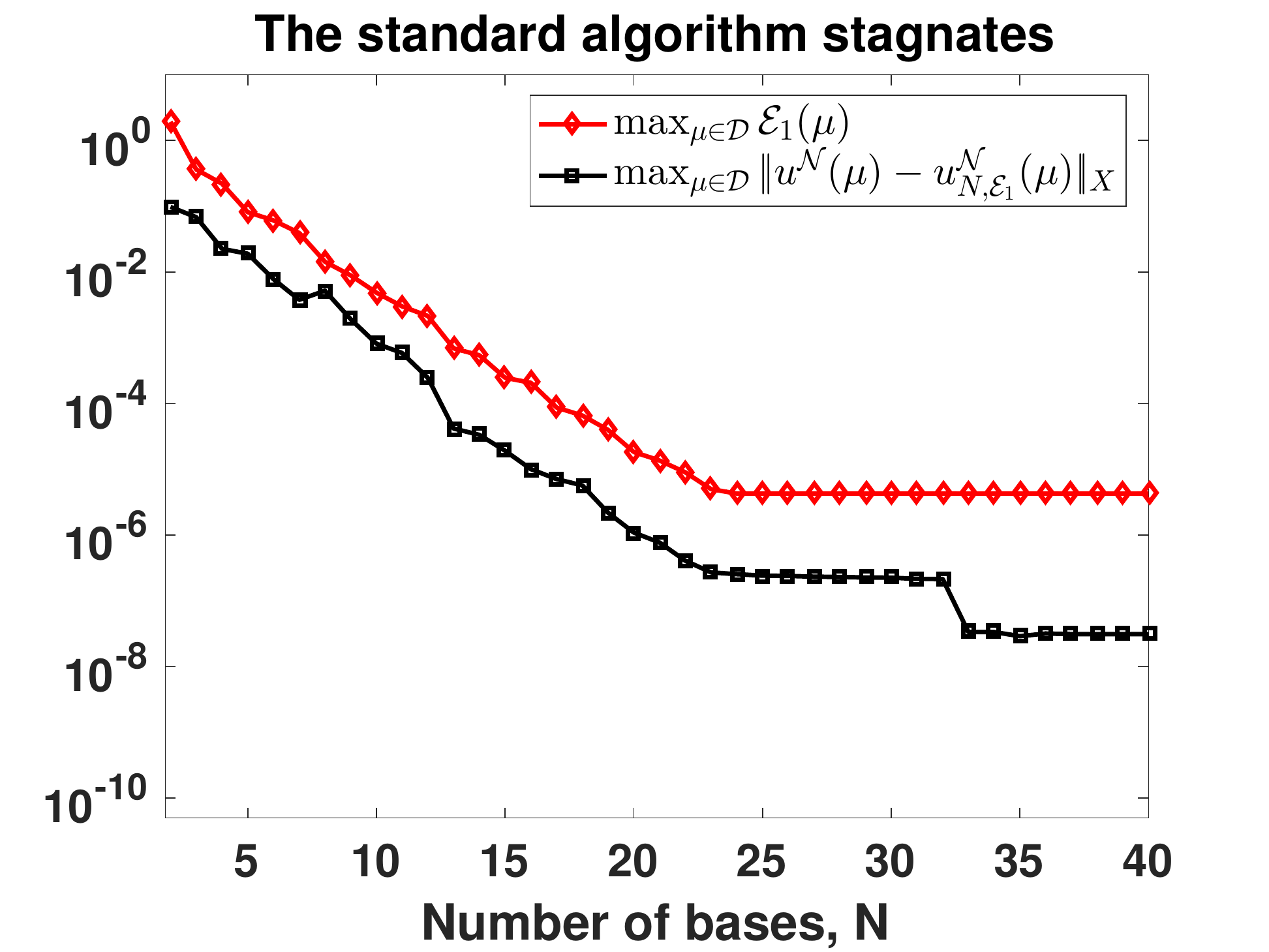}
  \includegraphics[width=0.49\textwidth]{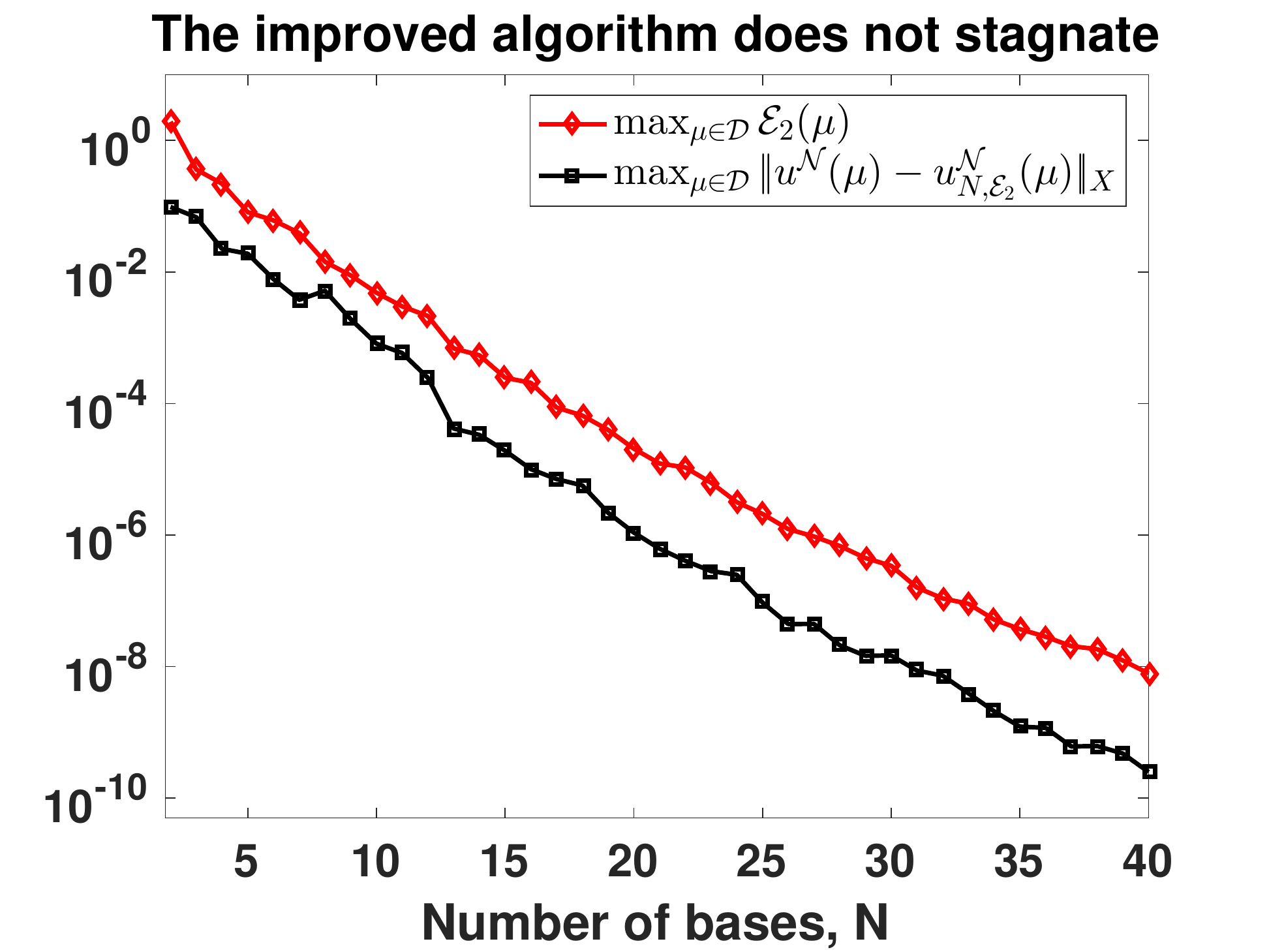}
  \includegraphics[width=0.49\textwidth]{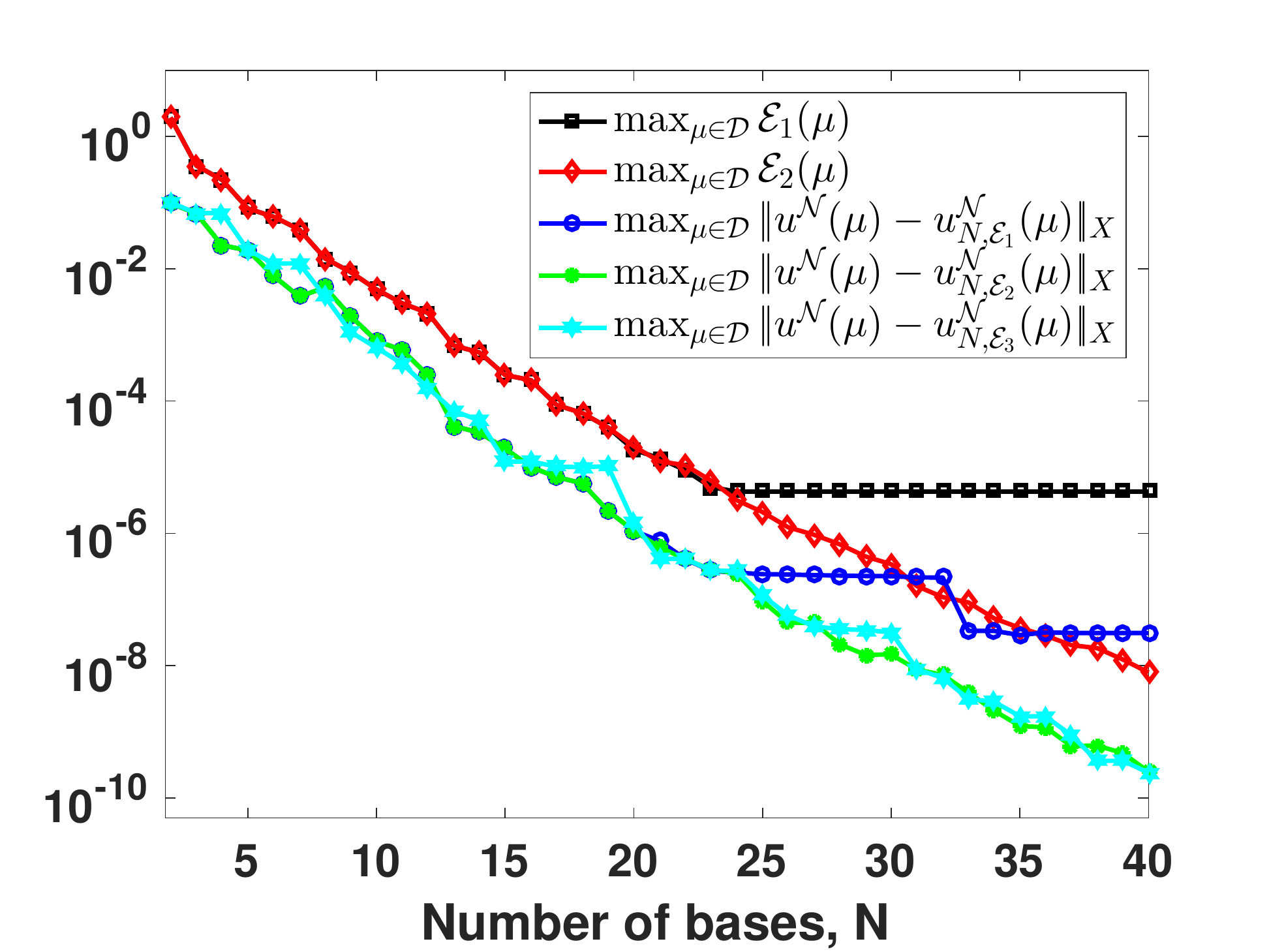}\\
  \includegraphics[width=0.49\textwidth]{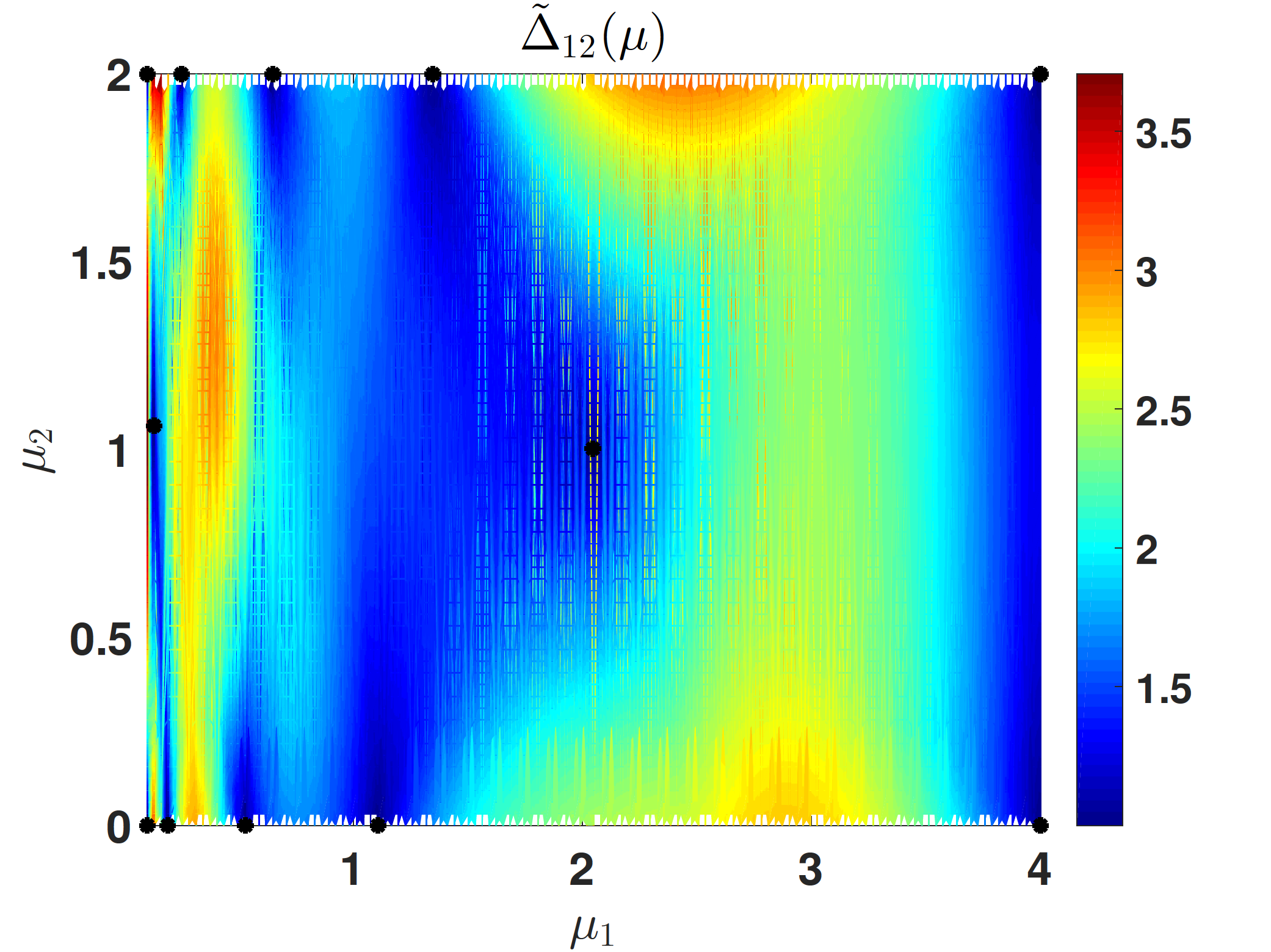}
  \includegraphics[width=0.49\textwidth]{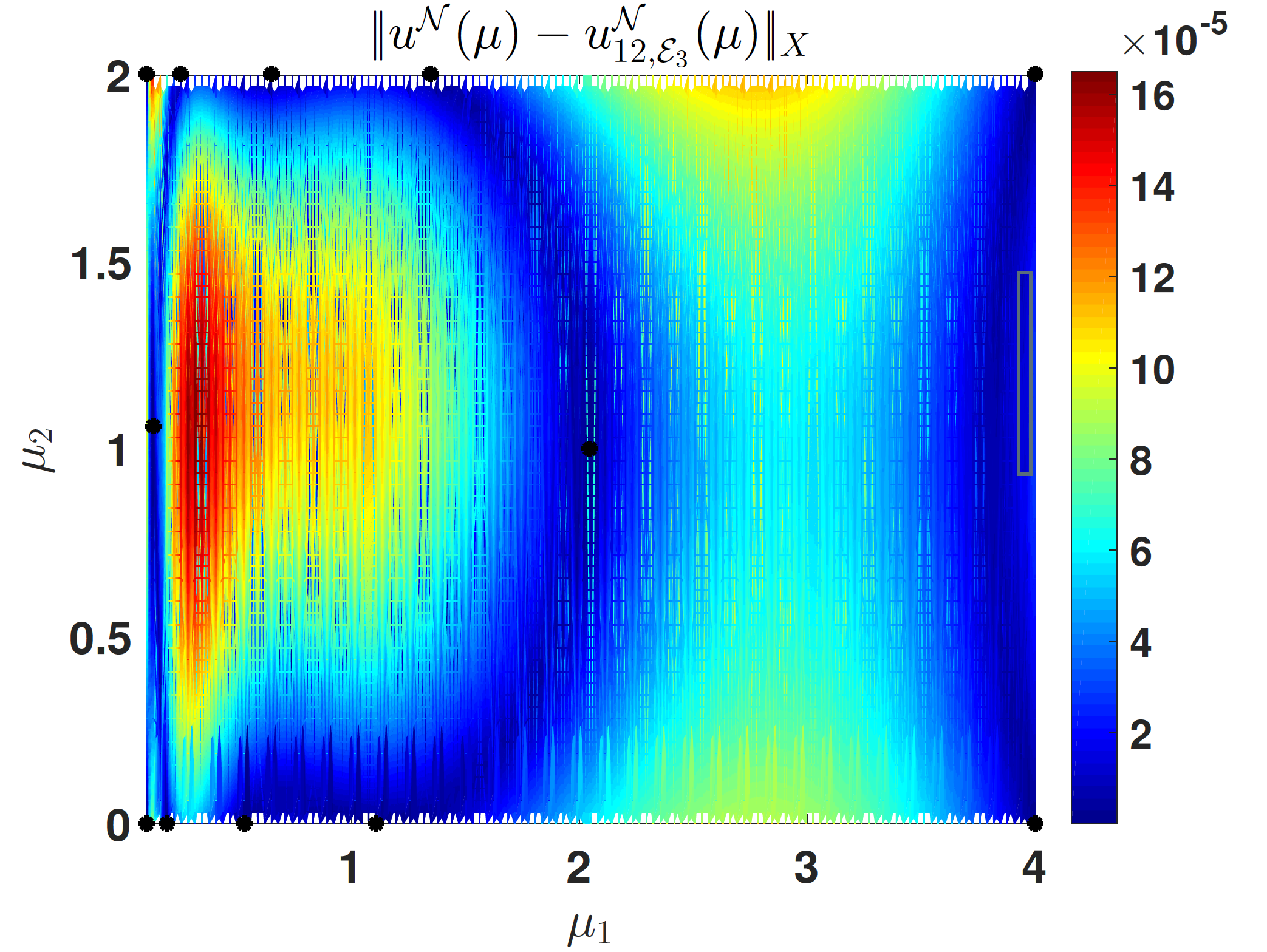}
\end{center}
  \caption{History of convergence for the classical and new approaches for equation \eqref{eq:2Ddiff_prob1}. The pictures on the bottom row indicate that residual-free error indicator {roughly} matches with the true error well. {Note that while error minima locations are successfully predicted, the maximum location(s) may be different.} }
\label{fig:2d_conv}
\end{figure}

We present the results in Figure \ref{fig:2d_conv}, and observe the same behavior as in the one-dimensional case. Namely, the error estimate and the true error stagnates when the traditional error estimate is adopted. This stagnation is eliminated { by the two} newly proposed approaches. 
What's more, the error indicator for the residual-free approach tracks the true error across the parameter domain as well as the one-dimensional case. 
It thus comes as no surprise that the third approach, albeit without a certificate, produces RB solution as accurate as the second method.

\begin{figure}
\begin{center}
  \includegraphics[width=0.9\textwidth]{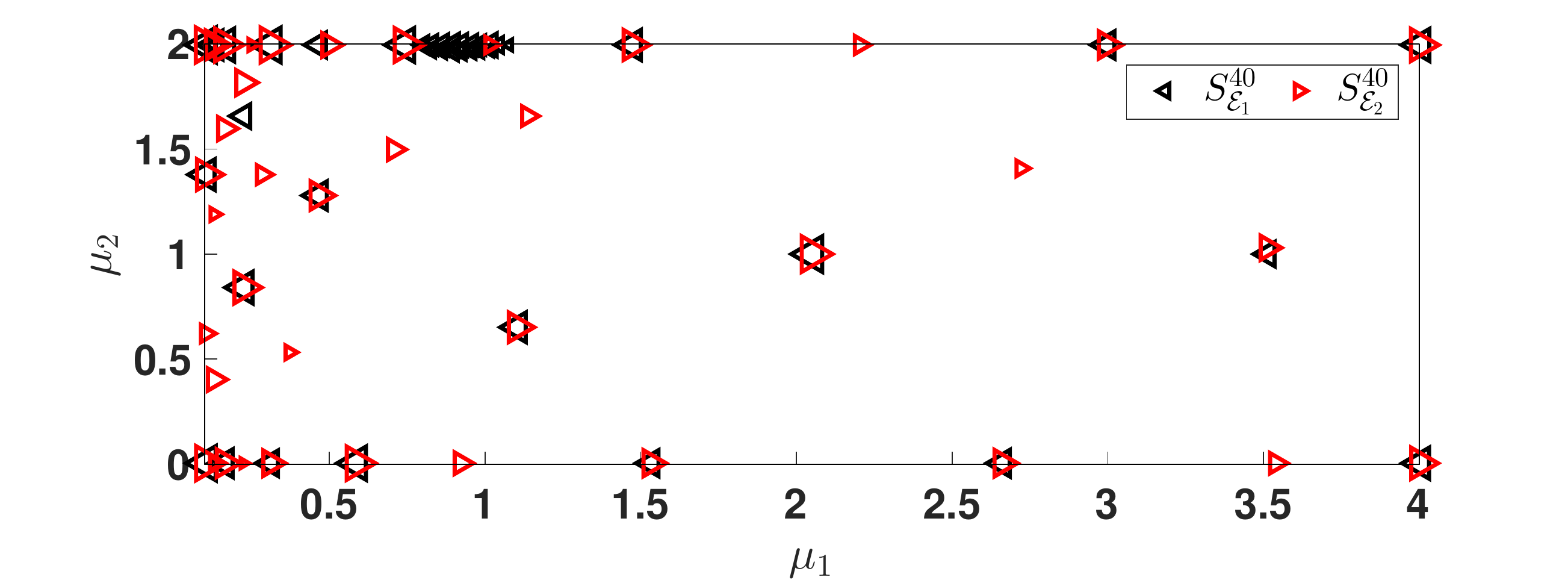}
  \includegraphics[width=0.9\textwidth]{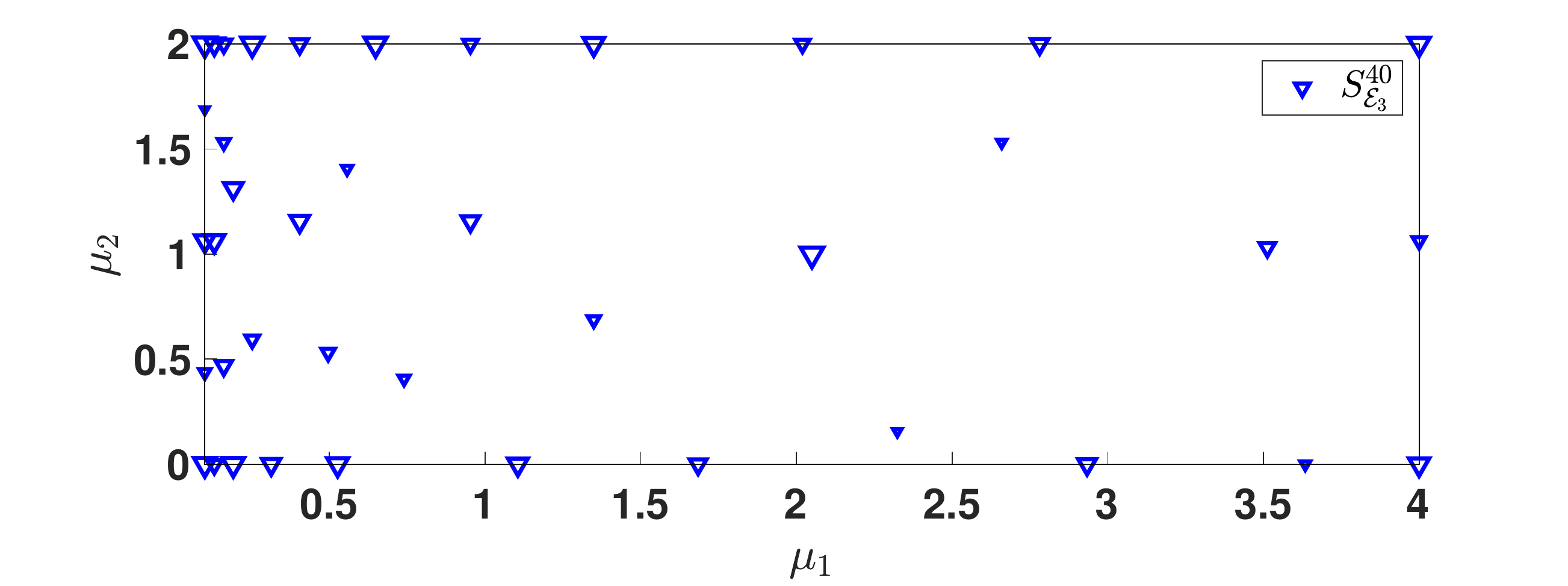}
  \includegraphics[width=0.9\textwidth]{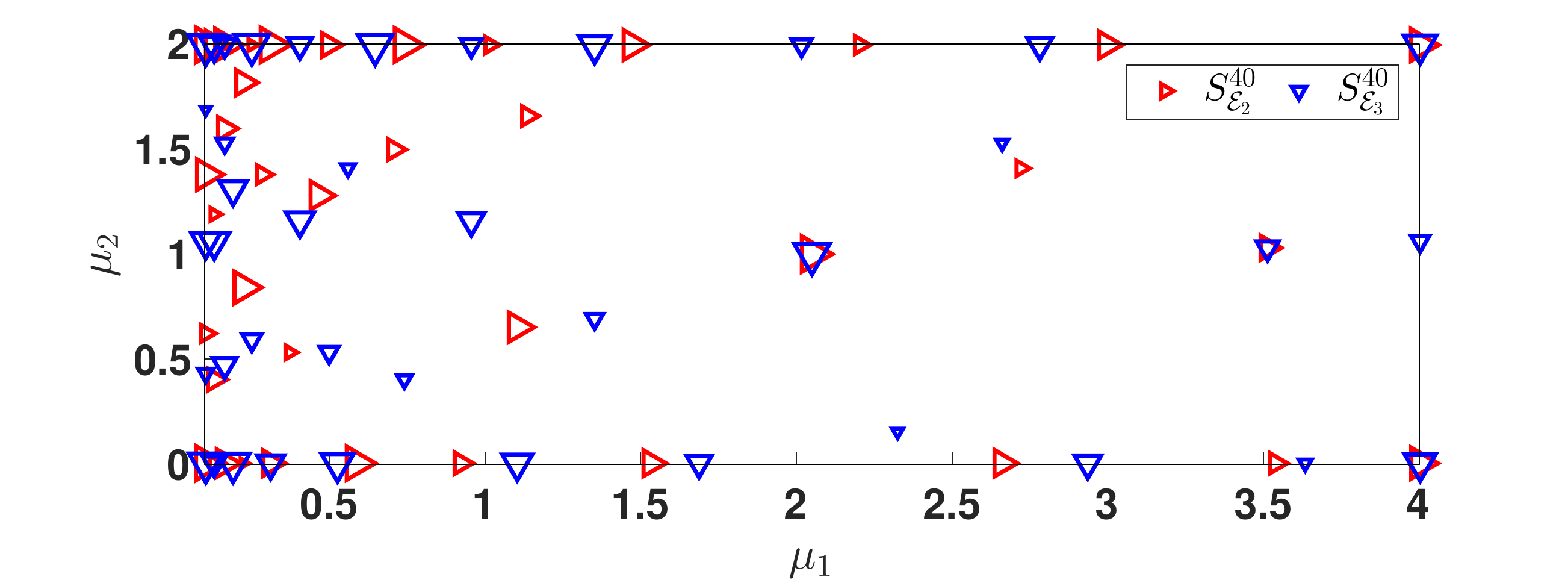}
\end{center}
\caption{Location of the chose snapshots for the classical approach and the two novel methods { for the second 2-dimensional test case}}
\label{fig:2d_points}
\end{figure}

We show in Figure \ref{fig:2d_points} the sets of snapshot locations $S^N$ for all three methods.  To differentiate them, we adopt the notation 
$S^N_{{\mathcal E}_i}$ for the sets produced by method $i$. We overlay the two sets $S^{40}_{{\mathcal E}_1}$ and $S^{40}_{{\mathcal E}_2}$ in the first picture. The larger the marker, the earlier it is picked by the greedy algorithm. Clearly, the two methods start by selecting the same points before deviating. The first method starts to clutter points in an unphysical manner and keeps doing so in the same neighborhood of the parameter domain. This leads to stagnation and potential singularity in the reduced solver. The enhanced approach ($S^{40}_{{\mathcal E}_2}$) avoids clustering and thus achieves better accuracy. The third, residual-free, approach demonstrates similar behavior.

\subsection{The second 2-dimensional test case}
The second two-dimensional example has an equation that induces a solution manifold that requires many more snapshots to achieve small error:
\begin{equation}
\label{eq:2Ddiff_prob2}
(1+\mu_{1}x)u_{xx} + (1+\mu_{2}y)u_{yy} = e^{4xy} \quad {\rm on} \quad \Omega.
\end{equation}
The parameter domain $\mathcal{D}$ here is taken to be $[-0.99, 0.99]^2$. The physical domain $\Omega$, boundary condition, truth solver and its resolution are all the same as the first two-dimensional case. We discretize $\calD$ using a tensorial $160 \times 160$ Cartesian grid with 160 equally-spaced points in each dimension. 

{The difficulty of this problem stems from the fact that the equation is close to degenerate at the four corners of the parameter domain. Thus the stability constant approaches zero toward the four corners, making accurate calculation of the residual norm even more critical. For example, the ratio \eqref{eq:error_estimator_en} blows up if the denominator (the stability constant) converges to zero while the numerator stays at the root machine accuracy level.}  The results are shown in Figure \ref{fig:results_p2} confirming, again, all previously stated properties for the two novel approaches. {The important role of an accurate residual norm calculation is apparent, as for example the chosen parameter values are unnecessarily more clustered toward the corners using the classical approach $\mathcal{E}_1$, see top row of Figure \ref{fig:results_p2}.}

\begin{figure}
\begin{center}
  \includegraphics[width=0.49\textwidth]{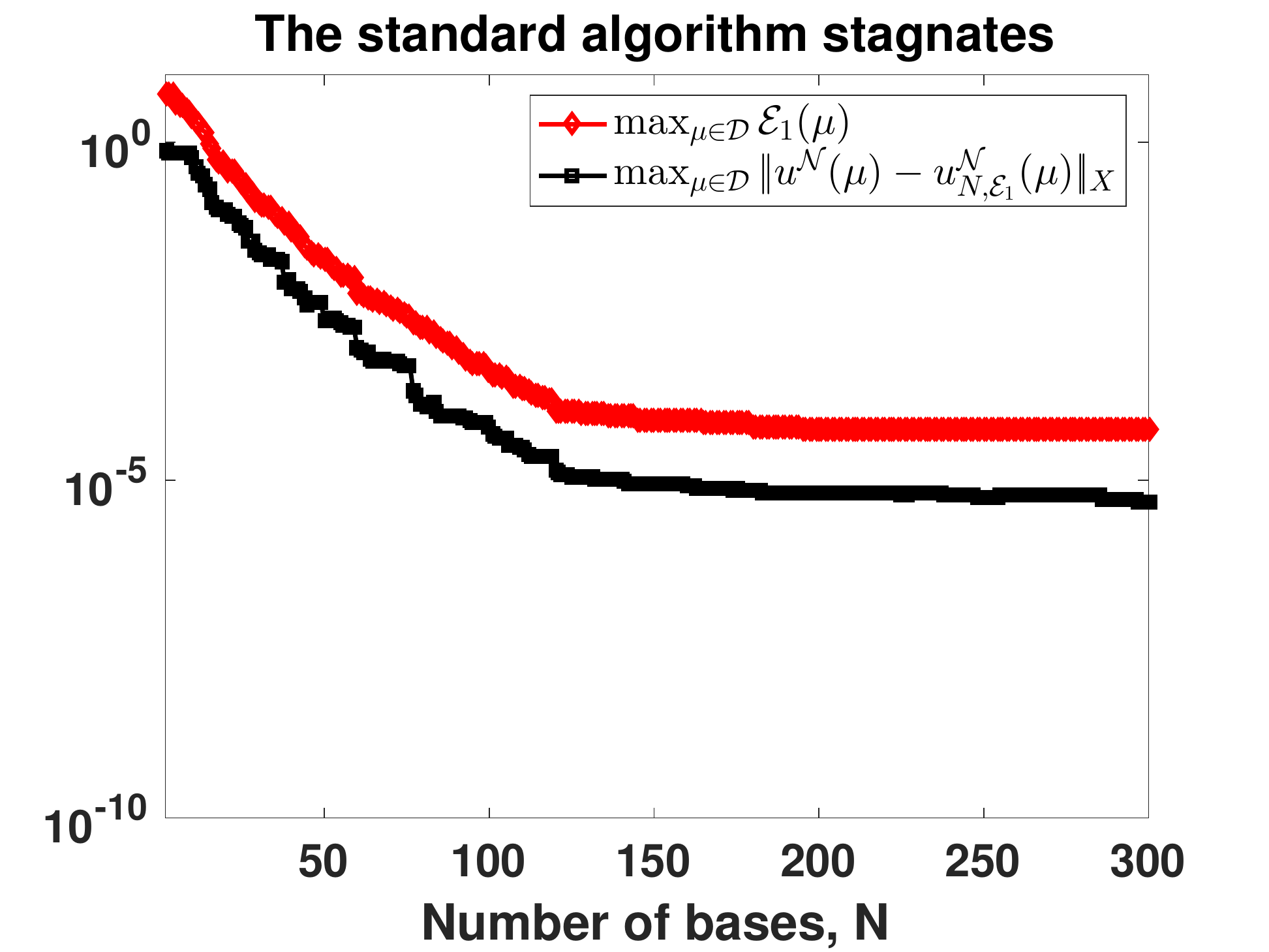}
  \includegraphics[width=0.49\textwidth]{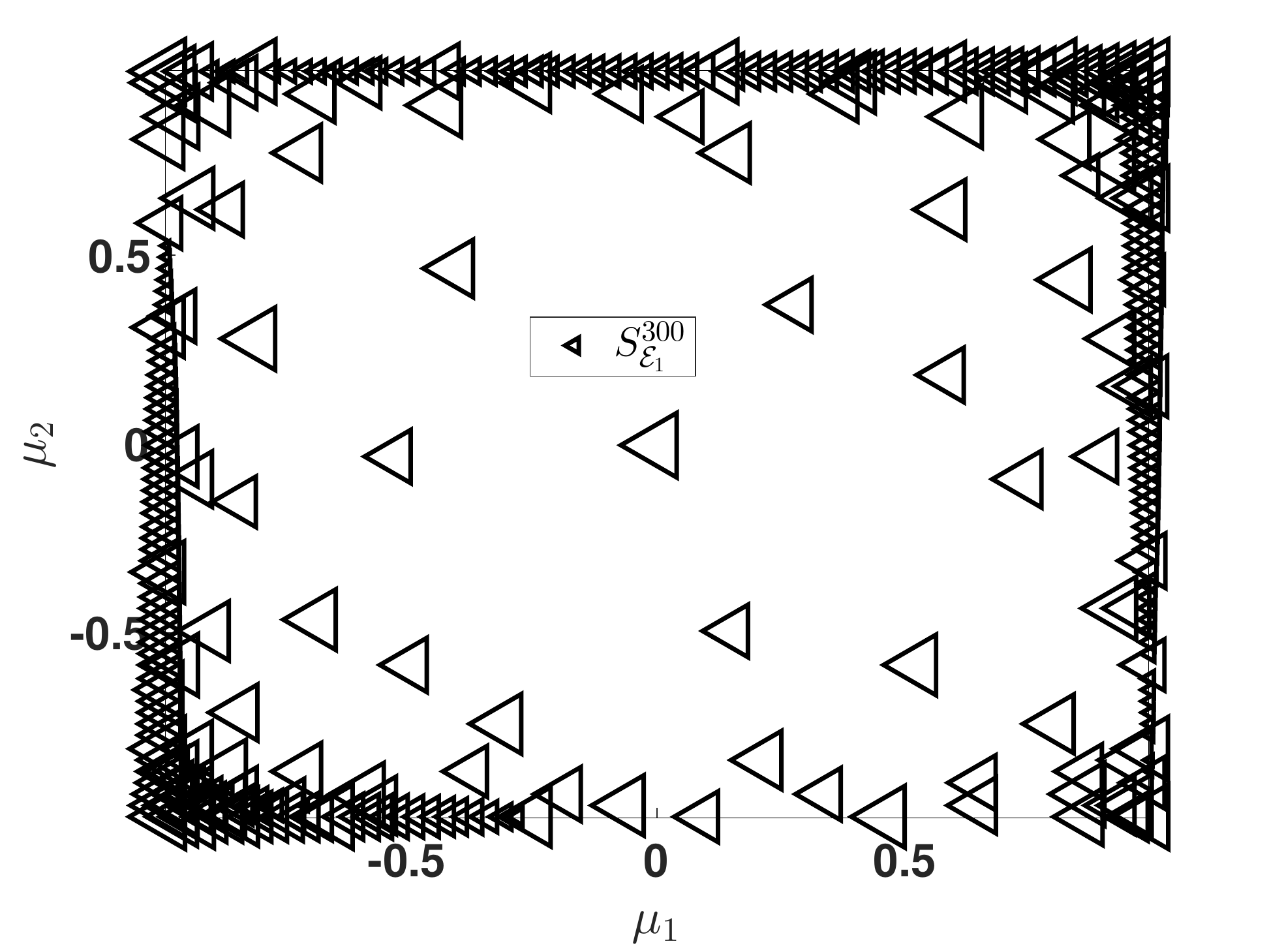}
    \includegraphics[width=0.49\textwidth]{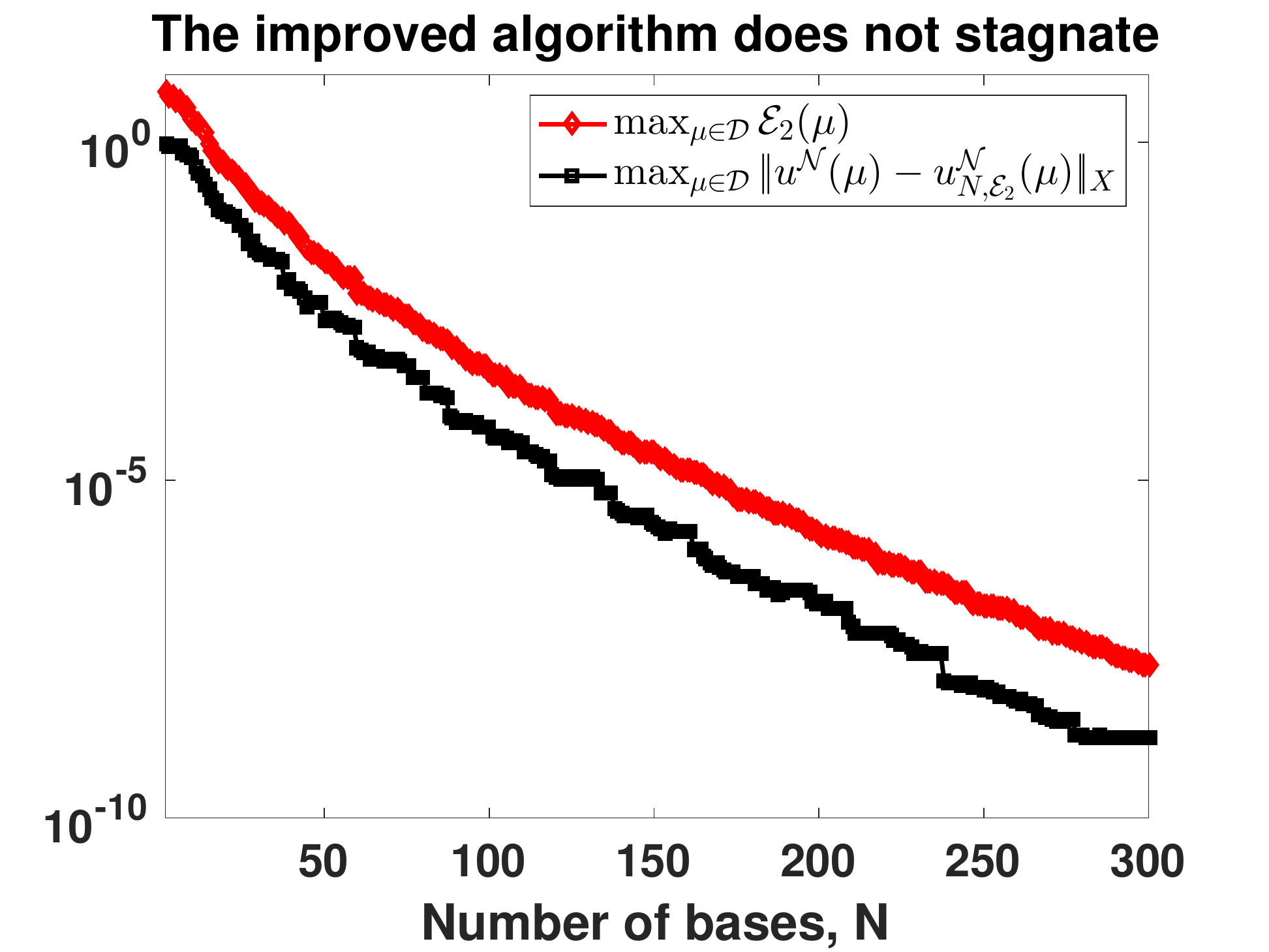}
  \includegraphics[width=0.49\textwidth]{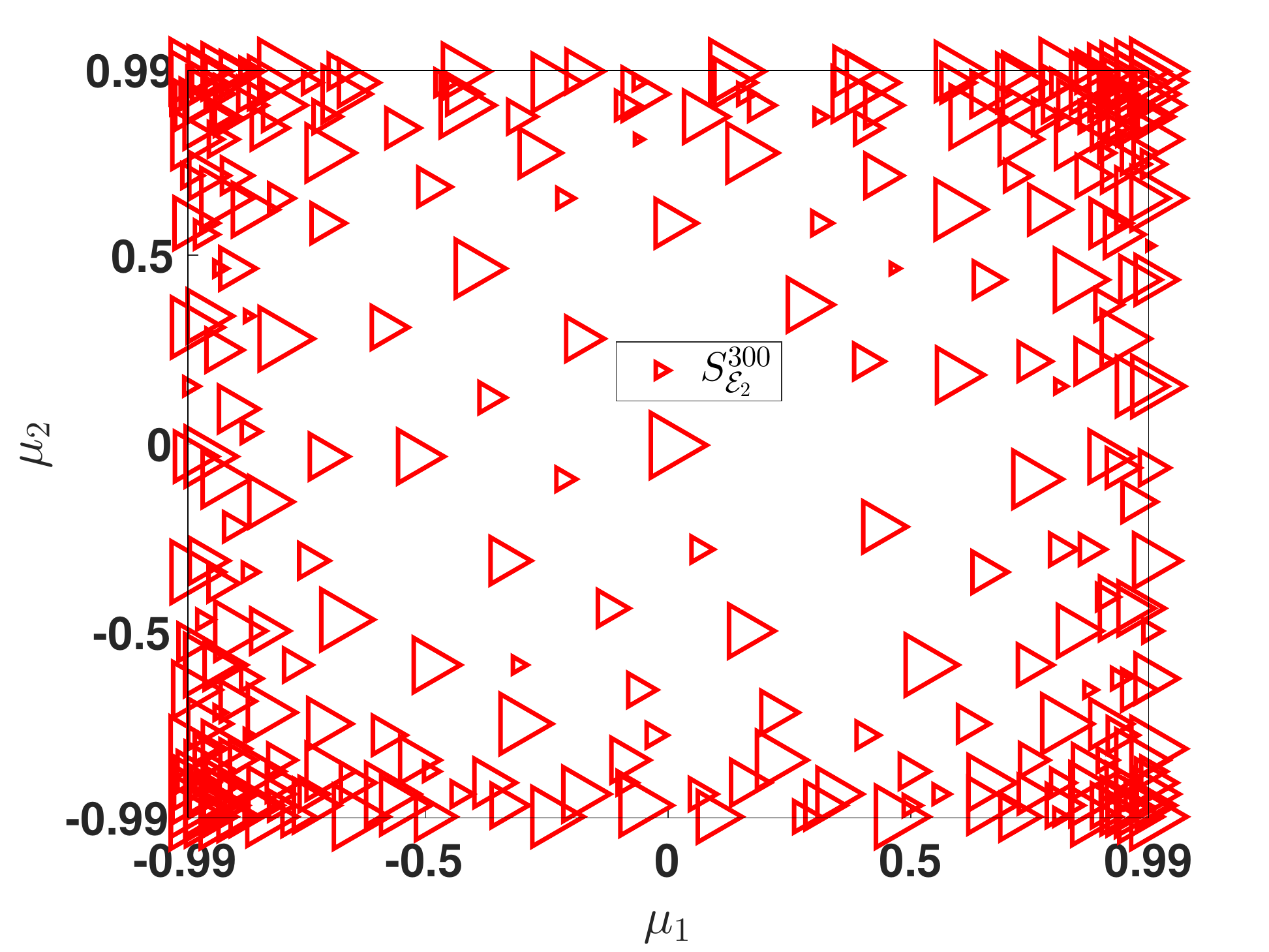}
    \includegraphics[width=0.49\textwidth]{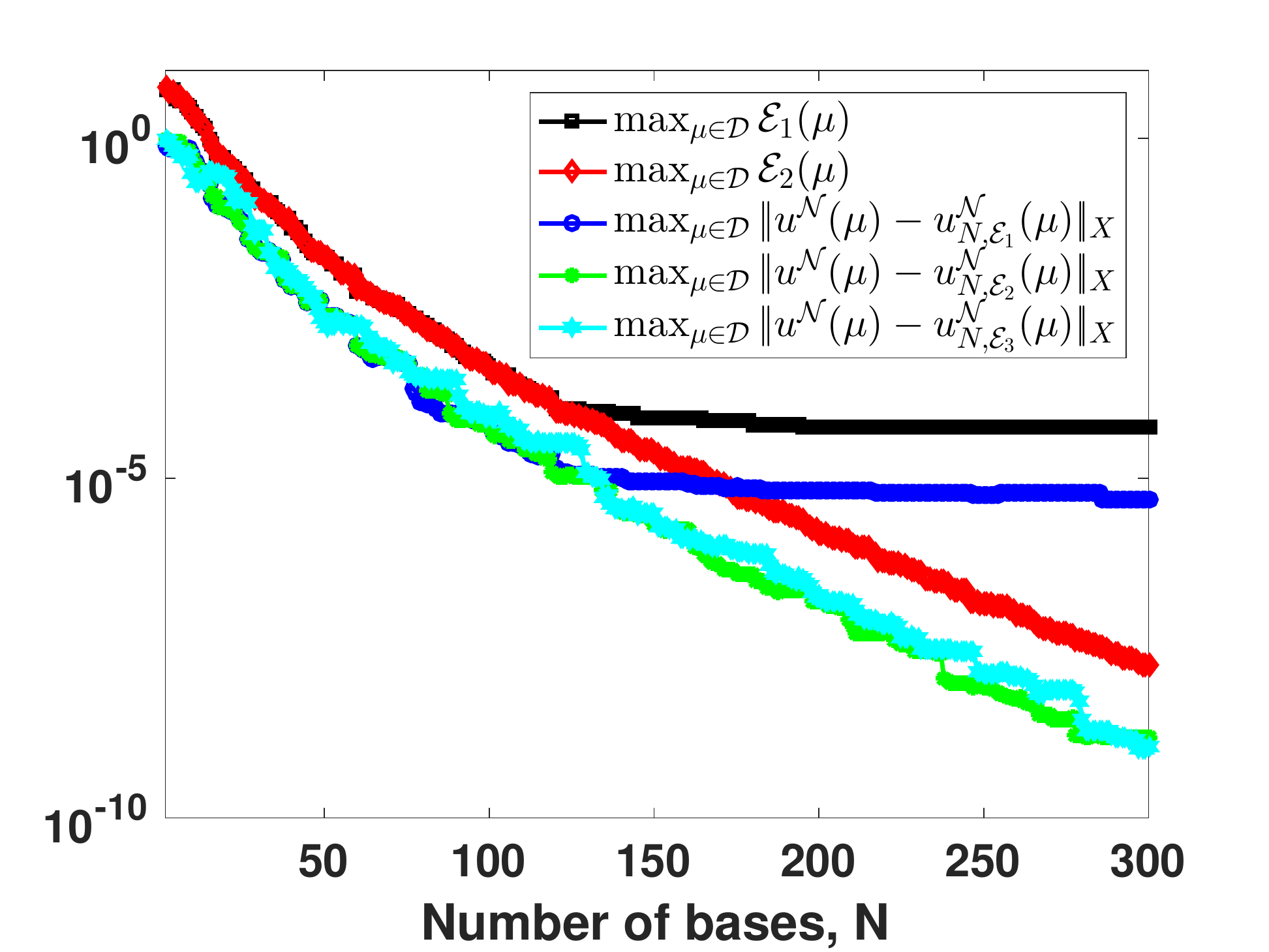}
  \includegraphics[width=0.49\textwidth]{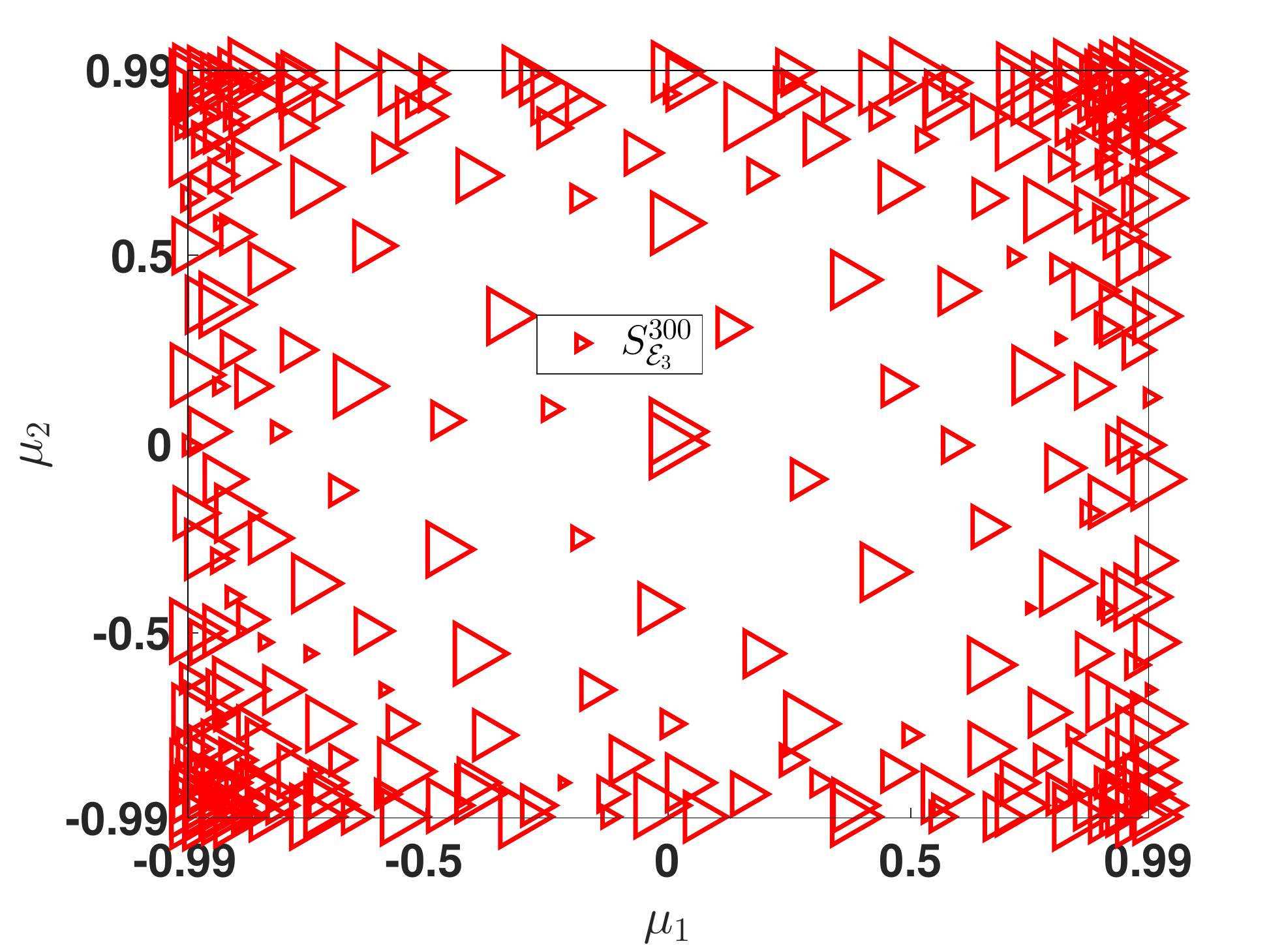}
\end{center}
\caption{The results for the second test problem.}
\label{fig:results_p2}
\end{figure}

\section{Concluding Remarks}
\label{sec:conclude}
We have proposed two novel strategies for computing objective functions in the offline greedy algorithm in the reduced basis method. Our first strategy is residual-based, and follows standard practice in RBM by defining the objective to be an \textit{a posteriori} upper bound for the error committed by a finite element method. This bound requires computation of a residual norm. In the standard RBM setting, this residual norm is computed in a way that can succumb to loss of significance when the magnitude of the norm reaches root machine precision. Our residual-based reformulation circumvents this premature stagnation without any loss in efficiency.

Our second strategy is residual-free, and uses only the RBM coefficients in the greedy objective. The particular function is the Lebesgue function from interpolation theory. We can provide a theoretical result guaranteeing that the parametric behavior of this function provides an upper bound for the parametric variation of the error, and thus is a good objective function for a greedy search. However, the precise connection between the Lebesgue function and the error involves a parameter-independent multiplicative constant that is an uncomputable best approximation error. Therefore, the residual-free method can effectively choose parameter values, but it cannot provide error certification without a quantitative understanding of this best approximation error. \annote{Furthermore, we currently lack a result establishing that the residual-free objective is a lower bound for the true error; the establishment of such a result is a subject of ongoing work.}

Our numerical experiments demonstrate that both of our strategies can effectively allow RBM to compute solutions to an accuracy much closer to machine precision than {the classical reduced basis error estimator}.

\bibliographystyle{abbrv}

\end{document}